%% file: NuccioOchiaiRay_v5.tex
\documentclass[11pt]{amsart}
\input{preamble}

\newtheorem{introtheorem}{Theorem}

\newtheorem{theorem}{Theorem}[section]
\newtheorem{lemma}[theorem]{Lemma}

\newtheorem{proposition}[theorem]{Proposition}
\newtheorem{corollary}[theorem]{Corollary}
\newtheorem{defn}[theorem]{Definition}

\newtheorem{hyp}[theorem]{Hypothesis}
\numberwithin{equation}{section}

\theoremstyle{remark}
\newtheorem{remark}[theorem]{Remark}

\newcommand{\cyc}{\mathrm{cyc}}
\newcommand{\Gal}{\mathrm{Gal}}
\DeclareMathOperator{\Ker}{Ker}
\DeclareMathOperator{\Coker}{Coker}

\DeclareMathOperator{\Hom}{Hom}
\newcommand{\Q}{\mathbb{Q}}
\newcommand{\Z}{\mathbb{Z}}
\newcommand{\N}{\mathbb{N}}
\newcommand{\Sel}{\mathcal{R}}
\newcommand{\OO}{\mathcal{O}_{\K}}

\newcommand{\Lcyc}{\Lambda_{\mathrm{cyc}}}
\newcommand{\resfield}{\kappa}
\newcommand{\cts}{\mathrm{cont}}
\newcommand{\tD}{\widetilde{\mathds{D}}}
\newcommand{\tL}{\widetilde{\AoK}}
\newcommand{\tT}{\widetilde{\TT}}
\newcommand{\D}{{\mathds{D}}}
\newcommand{\Y}{\mathcal{Y}}
\newcommand{\DFS}{\Y\big(\Q, (\tT)^\star(1)\big)}
\newcommand{\tCT}{\widetilde{T}}
\newcommand{\tchi}{\tilde{\chi}}
\newcommand{\tphi}{\tilde{\phi}}

\DeclareFontFamily{U}{wncy}{}
\DeclareFontShape{U}{wncy}{m}{n}{<->wncyr10}{}
\DeclareSymbolFont{mcy}{U}{wncy}{m}{n}
\DeclareMathSymbol{\Sha}{\mathord}{mcy}{"58}

\title[Coleman families and Iwasawa invariants]{A formal model of Coleman families and applications to Iwasawa invariants}

\author[F.~A.~E.~Nuccio]{Filippo A.~E.~Nuccio Mortarino Majno di Capriglio}
\address[F.~A.~E.~Nuccio]{Universit\'{e} Jean Monnet Saint-\'{E}tienne, CNRS, Institut Camille Jordan UMR 5208, \mbox{F-42023}, Saint-\'{E}tienne, France}
\email{filippo.nuccio@univ-st-etienne.fr}

\author[T.~Ochiai]{Tadashi Ochiai}
\address[T.~Ochiai]{Department of Mathematics \\ 
Tokyo Institute of Technology \\ 
2-12-1 Ookayama, Meguro-ku, Tokyo 152-8551, Japan}
\email{ochiai@math.titech.ac.jp}

\author[J.~Ray]{Jishnu Ray}
\address[J.~Ray]{Harish-Chandra Research Institute Chhatnag Road, Jhunsi, Prayagraj (Allahabad) 211 019 India }
\email{jishnuray@hri.res.in; jishnuray1992@gmail.com}



\begin{document}

\begin{abstract}
For a given Coleman family of modular forms, we construct a formal model and prove the existence of a family of Galois representations associated to the Coleman family. 
 As an application, we study the variations of Iwasawa $\lambda$- and $\mu$-invariants of dual fine (strict) Selmer groups over the cyclotomic $\Z_p$-extension of $\Q$ in Coleman families of modular forms. This generalizes an earlier work of Jha and Sujatha  for Hida families.
    \medskip

   \noindent\textsc{R\'esum\'e.} Nous construisons un mod\`{e}le formel pour toute famille de Coleman de formes modulaires, et nous montrons l'existence d'une famille de repr\'{e}sentations galoisiennes qui lui est attach\'{e}e. Gr\`{a}ce \`{a} cette construction nous \'{e}tudions la variation des invariants d'Iwasawa $\lambda$ et $\mu$ attach\'{e}s aux duaux des groupes de Selmer fins (strictes) le long de la $\Z_p$-extension cyclotomique de $\Q$ dans des familles de Coleman de formes modulaires. Nos r\'{e}sultats g\'{e}n\'{e}ralisent ceux obtenus par Jha et Sujatha dans le cas de familles de Hida de formes modulaires ordinaires.
\end{abstract}

\maketitle

\section{Introduction}
Throughout the article, we denote by $p$ a prime number greater than $2$, and we let $N$ be a positive integer that is prime to $p$.

In the celebrated papers \cite{Col97c} and \cite{Col97}, Coleman constructed $p$-adic families of elliptic cuspforms  with fixed slope, often referred to as \emph{Coleman families}. Coleman's theory is a non-ordinary generalization of Hida's theory, yet some of the properties of Coleman families are not completely parallel to those of Hida families. For example, Hida families are usually defined over the whole weight space, but Coleman families are usually defined only locally in the weight space. Also, a Hida family is defined over an Iwasawa algebra, or over a complete and semi-local ring which is finite over an Iwasawa algebra, 
whereas a Coleman family as constructed in~\cite{Col97} is defined over an affinoid algebra. 
In particular, Wiles' method of pseudo-representations does not apply \emph{verbatim} to attach a family of Galois representations to a Coleman 
family because affinoid algebras are not semi-local. Yet, this is the more natural setting to study Iwasawa-theoretic questions, that are the main focus of our work. Our aim is to analyse the behaviour of some Iwasawa invariants in Coleman families, and for this we need a convenient algebraic framework to simultaneously treat the ``big'' $p$-adic Galois representations over the whole Coleman family, as well as the ``specialized'' ones attached to each cuspform in the family. Moreover, the study of integrality questions, or growth of denominators, is crucial for defining Iwasawa's $\mu$-invariant and this analysis makes sense only over integral domains where $p$ is not invertible. Thus we devote the first half of the paper to detailing a construction of an integral formal model for Coleman families over a complete local 
$\mathcal{O}_{\mathcal{K}}$-algebra ${\AoK}$ where $\mathcal{O}_{\mathcal{K}}$ is the ring of integers of a finite extension $\mathcal{K}$ of $\mathbb{Q}_p$. We do not recall a precise definition of ${\AoK}$ here, rather pointing the reader to~\eqref{equation:A0}. The spectrum of ${\AoK}$ contains a certain countable infinite set $\Zf$ which is identified with a subset of $\mathbb{Z}_{\geq 2}$ (see~\eqref{equation:definitionZf} for the definition of $\Zf$), so that for every $k\in\Zf$ we can consider the unique classical cuspform $f_k$ that belongs to the Coleman family---or that passes through it, as one can say from a more geometrical viewpoint. Denote by $\rho_k$ the continuous, two-dimensional $\Q_p$-representation attached by Deligne to $f_k$: our first result goes in the direction of $p$-adically interpolating these Galois representations. Explicitly, we prove the following theorem that attaches a ``big'' Galois representation to the given Coleman family by the method of pseudo-representation: we refer to Theorem~\ref{BigGaloisRep} for the precise statement and the proof, here it suffices to say that $S$ is a finite set of rational primes containing $p$, and $P_k$ is the prime ideal in $\AoK$, corresponding to the point of $\operatorname{Spec}{\AoK}$ identified with the integer $k$, and such that $\AoK/P_k$ is a finite $\Z_p$-module.
\begin{introtheorem}\label{introtheorem:NuccioOchiai}
There exists a free ${\AoK}$-module $\TT$ of rank two with a continuous action of unramified-outside-$S$ Galois group $G_{\Q,S}$ such that the representation 
\[
\boldsymbol{\rho}\colon G_{\Q,S} \longrightarrow\mathrm{Aut}_{{\AoK}}(\TT )
\]
satisfies $\pi_{\boldsymbol{\rho}}=\pi$ (in a suitable basis). In particular, $\boldsymbol{\rho}$ modulo the ideal 
$P_k$ corresponding to $k$ is isomorphic to a lattice of $\rho_k$ for all $k\in \Zf$.
\end{introtheorem}
The construction of a big Galois representation attached to a Coleman family is not new. Already in the foundational articles~\cite{Col97} and~\cite{ColMaz98} the authors described it, and another approach (based on the theory of modular symbols) was later proposed by Bella\"{i}che in~\cite{Bel12}. In these works, the Galois representation is supported on a rank-$2$ module over a Tate algebra, or possibly over the subalgebra of power-bounded elements in such a Tate algebra. This ring is not semi-local and therefore these constructions were not suited for our specialization techniques for fine Selmer groups in the setting of Iwasawa theory. The definition of a formal model for the Galois representation attached to a Coleman family that we present in this paper was first introduced in the preprint~\cite{NucOch16}, that has received some citations but has finally not been published. We take the opportunity to present the construction here, together with its Iwasawa-theoretic applications. Independently from ours, another construction of a formal model is explained in~\cite[Remark~1.4 and Remark~A.1]{BL21} and relies on Loeffler--Zerbes' work~\cite{LoeZer16}, where the representation arises in the cohomology of ``big'' \'{e}tale sheaves on modular curves as defined in~\cite{AndIovSte15}. Our approach differs from theirs, as it mainly relies on the theory of pseudo-representations as in Wiles' original work~\cite{Wil88} (see also Hida's monograph~\cite{Hid93}). It has the advantage that it does not use any specific context of elliptic modular forms so that it can be generalised to the setting of Coleman families associated to more general Shimura varieties.

In the second half of this article, we analyse some Iwasawa-theoretic properties of the family of Galois representations constructed in the above theorem. This generalizes results of Greenberg--Vatsal (see~\cite{GreenbergVatsal}) and of Emerton--Pollack--Weston (see~\cite{EPW}), who studied the variations of  Iwasawa $\lambda$- and $\mu$-invariants of dual Selmer groups over the cyclotomic $\Z_p$-extension of $\Q$ for families of $p$-ordinary modular forms. More precisely, let $\bar{r}\colon G_\Q \rightarrow \operatorname{GL}_2(\resfield)$ be an absolutely irreducible modular Galois representation of $G_\Q$ over a finite field $\resfield$ of characteristic $p$. Assume that $\bar{r}$ is $p$-ordinary and $p$-distinguished in the sense that the restriction of $\bar{r}$ to a decomposition subgroup at $p$ is reducible and non-scalar (see \cite[p.~523]{EPW}), and let $\mathcal{H}(\bar{r})$ be the Hida family of~$\bar{r}$. Emerton--Pollack--Weston showed that if the $\mu$-invariant for the dual Selmer group for some $f_0 \in \mathcal{H}(\bar{r})$ is trivial, then it is so for all $f \in \mathcal{H}(\bar{r})$. Further, they showed that, under the assumption that the $\mu$-invariant is zero, the $\lambda$-invariant is constant across a branch in the Hida family $\mathcal{H}(\bar{r})$. 

Our goal is to study an Emerton--Pollack--Weston type result for the \emph{fine} Selmer group in a Coleman family. The fine Selmer group is a much studied object in Iwasawa theory and occurs in the formulation of the Iwasawa Main Conjecture for elliptic curves and modular forms.
In \cite{CoatesSujatha_fineSelmer}, Coates and Sujatha initiated a systematic study of the fine Selmer group and formulated conjectures about its structure over a $p$-adic Lie extension. Under certain hypotheses, they have shown that Iwasawa's $\mu$-invariant Conjecture for the cyclotomic extension of a number field is equivalent to Conjecture A on the dual fine Selmer group of an elliptic curve (see \cite[Theorem 3.4]{CoatesSujatha_fineSelmer}). Conjecture A says that the $\mu$-invariant of the dual fine Selmer group of an elliptic curve over the cyclotomic $\Z_p$-extension of a number field is trivial.
Observe that the Pontryagin dual of the fine Selmer group over the cyclotomic $\Z_p$-extension of $\Q$ is proved to be torsion by a  deep result of Kato \cite[Theorem~12.4]{Kato} regardless of whether $f$ is ordinary at $p$ or not. An interested reader might also want to view the Iwasawa Main Conjecture framed in terms of dual fine Selmer group; see \cite[Conjecture~6.1]{kurihara02} or \cite[Proposition~7.1 (ii)]{Kobayashi} and \cite[Conjecture~12.10]{Kato}. This is the version of  the Iwasawa Main Conjecture without $p$-adic $L$-functions, framed in terms of fine Selmer groups, and this is what is usually considered when dealing with non-ordinary families (see \cite[Section~5]{NakamuraK}). This justifies the use of fine Selmer groups in this paper; the construction of plus and minus Selmer groups over a Coleman family  is unattainable by known techniques.

Variations of Iwasawa invariants for the dual fine Selmer group in Hida families of ordinary modular forms have been studied by Jha and Sujatha in~\cite{Somnath}. In this article, we go beyond their work and study these variations for the dual fine Selmer group when the modular forms vary in a Coleman family. We define a universal ``big'' fine Selmer group for the Galois representation $\TT$ produced in Theorem~\ref{introtheorem:NuccioOchiai}, and we show that it is cotorsion over a two-variables Iwasawa algebra. Further, we specialize this big fine Selmer group at classical points and study its relation with the classical fine Selmer group (see Theorem~\ref{thm:specialization}). In this context, our main result is the following, where $\boldsymbol{\rho}$ is as in Theorem~\ref{introtheorem:NuccioOchiai}:

\begin{introtheorem}[see Theorems \ref{thm:main_mu} and \ref{thm:lambda}]\label{introtheorem:Ray}
Assume that the residual representation of $\boldsymbol{\rho}$ when restricted to $G_{\Q_p}$ is irreducible. Given any $k\in\Zf$, the $\mu$-invariant of the dual fine Selmer group associated to $\rho_k$ is zero if and only if the same is true for any other classical modular form in the Coleman family. Under the assumption that the $\mu$-invariant of one, or of any, form is zero, the $\lambda$-invariants are constant for all but finitely many classical modular forms in the Coleman family.
\end{introtheorem}

Other authors have studied the question of the variation of Iwasawa invariants in families of congruent modular forms. In \cite{Kim09}, Kim studied Iwasawa $\lambda$- and $\mu$-invariants for  the dual  signed Selmer groups over the cyclotomic $\Z_p$-extension of $\Q$ for elliptic curves at supersingular primes. He proved that they are constant for a family of elliptic curves with the same residual representation if the $\mu$-invariant of any of them is zero. Kim's result was later generalized by the first-named author and Sujatha in~\cite{NucSuj21} in a more general, but still supersingular, setting. Under Conjecture A and a few other additional assumptions  they could give a criterion when the $\mu$-invariant of the dual signed Selmer group vanishes purely in terms of a signed \textit{residual} Selmer group (see \cite[Theorem~4.13]{NucSuj21}).
Variations of Iwasawa invariants for dual  fine Selmer groups were also studied by Kim, Lee and Ponsinet in~\cite{KLP} but there the authors considered a family of modular forms with fixed weight and varying tame levels, contrary to our set-up. Hatley and Lei considered the case of the variations of Iwasawa $\lambda$- and $\mu$-invariants for dual signed Selmer groups for non-ordinary congruent modular forms and proved a result similar to Kim's (see \cite[Theorem~4.6]{HatleyLei}). In \cite{LimSujatha}, Lim and Sujatha studied fine Selmer groups of two congruent Galois representations over an admissible $p$-adic Lie extension and showed that, under appropriate congruence conditions, if the dual fine Selmer group of one form is pseudo-null, the same holds for the other form (see \cite[Theorem~3.7]{{LimSujatha}}).

The structure of the paper is as follows. In Section~\ref{sec:affinoids}, we prepare some generalities on affinoid algebras and on the weight space in relation to Coleman families. Based on this preparation, we construct a formal model of a given Coleman family in Section~\ref{section:family_of_modular_forms}. 
In Section~\ref{sec:biggaloisrep}, we explain the construction of the Galois representation giving the proof of Theorem~\ref{introtheorem:NuccioOchiai} (in the precise formulation of Theorem~\ref{BigGaloisRep}). In Section~\ref{sec:deformations}, we review basic facts on  Galois deformations of Coleman families and construct the big fine Selmer group associated to a Coleman family. The specialization theorem at an arithmetic prime is given in Section~\ref{sub:specialization}. We use this specialization theorem to show that the big fine Selmer group is cotorsion over a two-variables Iwasawa algebra. Our results about variations of the Iwasawa invariants, summarized in Theorem~\ref{introtheorem:Ray}, are obtained in Section~\ref{sub:maisection}.

\section*{Acknowledgement}
 The authors thank Adrian Iovita and Somnath Jha for answering several questions and for their comments on an earlier draft of this article. We warmly thank the two referees for a very carefully reading and for suggesting several corrections that helped improving the presentation and the mathematics of our manuscript. The third author is supported by the Inspire research grant.

\section{Preliminaries on affinoids and weight spaces}\label{sec:affinoids}
We refer the reader to \cite{BosGunRem84} for our conventions and basic results about rigid analytic spaces in the sense of Tate. Recall that $p$ is a fixed odd prime and $N\geq 1$ satisfies $(p,N)=1$.

Let $K$ be a complete subfield of $\C$. The field $K$ can be either 
a finite or an infinite extension of $\Q_p$.  
Let $\X$ be an affinoid space defined over $K$. We write $\R_\X$ for the ring of analytic functions on $\X$ and $\OOO_{\X}$ for the subring of power-bounded elements (see \cite[Section~1.2.5]{BosGunRem84}). They will always be endowed with their Gau\ss \ 
semi-norm (which is a norm and coincides with the $\sup$-norm if $\X$ is reduced). 
When $K$ is a discrete valuation field, the ring of power-bounded elements $\OOO_{\X}$ 
is noetherian because it is a quotient of $\mathcal{O}_K \langle T_1,\hdots,T_n\rangle$ 
and $\mathcal{O}_K \langle T_1,\hdots,T_n\rangle$
is the $p$-adic completion of a polynomial algebra
over $\mathcal{O}_K$. 
The ring $\R_\X=\OOO_\X[\frac{1}{p}]$ is noetherian whether $K$ is a discrete valuation 
field or not. For every maximal ideal $\mathfrak{m}\subseteq \R_\X$, the quotient
$\R_\X/\mathfrak{m}$ is a finite extension of $K$ by \cite[Section~6.1.2, Corollary 3]{BosGunRem84} 
and $\OOO_{\X}/\mathfrak{m}^\circ$ is a domain which is finite over $\mathcal{O}_K$ 
by \cite[Section~6.1.3, Proposition 3]{BosGunRem84}.
  
We consider the following definition:
\begin{defn} 
Let $K$ be a complete subfield of $\C$ 
and let $\X$ be an affinoid space over $K$. 
A subset $\ZF$ of the set of $K$-valued points $\X(K)$ 
is said to be Zariski-dense in $\X$ if we have $U(K) \cap \ZF\neq\emptyset $
for every non-empty Zariski-open subspace $U\subseteq \X$. 
\end{defn}
Given $x_0\in K$ and $r\in p^{\mathbb{Q}}$, 
we denote by $\B[x_0 ,r]_K$ and $\B(x_0 ,r)_K$, respectively, the closed and open ball of radius $r$ and centre $x_0$, 
seen as $K$-rigid analytic spaces (see \cite[Section~7]{deJ95} for a description of the second space). We note that we normalize the $p$-adic 
absolute value $\vert \ \ \vert$ so that 
$\vert p \vert  = \frac{1}{p}$. 
For example, in the case $r=1$ and $x_0 \in K$, $\OOO_{\B[x_0,1]_K}$ is isomorphic to 
the following ring of restricted power series with coefficients in $\mathcal{O}_K$: 
$$ 
\mathcal{O}_K \langle T-x_0 \rangle =
\Big\{\sum^\infty_{i=0}c_i (T-x_0)^i \in \mathcal{O}_K [\![ T-x_0 ]\!] 
 \ \Big\vert \ \lim_{i\to\infty}\vert c_i\vert =0\Big\} . 
$$
Finally, given any complete subfield $L\subseteq\mathbb{C}_p$, we also need the notation $B[a,r]_L$ and $B(a,r)_L$ for the \emph{set }of all $x\in L$ such that $\vert x-a\vert \leq r$ (respectively, such that $\vert x-a\vert <r$). When $K=\Q_p$, we denote $B[x_0,r]_{\Q_p}$ (resp.~$B(x_0,r)_{\Q_p}$) by $B[x_0,r]$ (resp.~$B(x_0,r)$) 
dropping the subscript.
\begin{lemma}\label{basic:Zariski-dense}
Let $K$ be a complete subfield of $\C$ which is a discrete valuation field. 
\begin{enumerate}[label=\textup{(}\arabic*\textup{)}]
\item When $\X$ is a reduced affinoid defined over $K$ and $f\in \R_\X$ vanishes on every point of a Zariski-dense subset $Z$, we have $f=0$.
\item Let $x_0 \in K$. Every infinite set inside $\B[x_0,1]_K (K)$ is Zariski-dense.
\end{enumerate}
\end{lemma}
\begin{proof}
For the first assertion, suppose $f\not= 0$ and consider the Zariski-open subset $U_f =\{x\in \X\text{ such that }f(x)\neq 0\}$ of $\X$. 
Since $\X$ is reduced and $f\neq 0$, we have $U_f\neq\emptyset$. 
By the assumption that $Z \subset \X$ is a Zariski-dense subset, we have $Z\cap U_f(K) \neq\emptyset$.
For any point $z\in Z\cap U_f(K)$, we have $f(z)=0$, contradicting the definition of 
$U_f$.
 
We pass to the second assertion. By the Weierstra{\ss} Preparation Theorem 
(\cite[Section~5.2.2, Theorem~1]{BosGunRem84}), 
every function $f\in \R_{\B[0,1]_K}$ can be factored as $f=P\cdot U$ where $P\in K 
[T-x_0 ]$ is a polynomial and $U\in \R_{\B[x_0,1]_K}^\times$ is an invertible power series which does not vanish on $\B[x_0,1]_K$. It follows that every such $f$ has only finitely many zeroes and that $\R_{\B[x_0,1]_K}$ is a PID, showing that non-trivial Zariski-closed sets in $\B[x_0,1]_K$ consist of finitely many points.
\end{proof}
One of the main rigid spaces of interest for us is the weight space $\W$, which is isomorphic to $\varphi(Np)$ copies of $\B(1,1)_K$ indexed by
\[
\dualD=\mathrm{Hom}\big((\mathbb{Z}/Np\mathbb{Z})^\times,\mathbb{C}_p^\times\big) .
\]
For generalities about $\W$, we refer to~\cite[Section~B1]{Col97} and to~\cite[Section~1.4]{ColMaz98}. 
For more detailed accounts, we refer to~\cite[Chapter~I.3, \S4 and Appendix]{Gou88} and~\cite[p.~103]{Buz07}. By definition, the weight space satisfies 
\[
\W(\C )=\operatorname{Hom}_\mathrm{cont}(\varprojlim_n(\mathbb{Z}/Np^n\mathbb{Z})^\times,
\C^\times) .
\]
Following Coleman and Mazur, we give the following definition: 
\begin{defn}\label{def:weightSp} 
Denote by $\omega\colon\mu_p(\Z_p)\to \Z_p^\times$ the Teichm\"{u}ller character and denote the 
the projection $x\mapsto x/\omega(x)$ by $\langle\!\langle \ \rangle\!\rangle\colon\Z_p^\times\to 1+p\Z_p$. For every integer $k\in\Z$ and every character $\chi\in\dualD$ of finite order, the point 
$\chi\langle\!\langle \ \rangle\!\rangle^k\in\W(\Q_p )$ is called an \emph{accessible} weight-character with coordinates $(\chi,k)$.
\end{defn}
As detailed in \cite[Definition~in~Section~1.4]{ColMaz98} the accessible weight-characters are parametrized by the rigid analytic subspace $\W^\ast=\dualD\times \B^\ast\subseteq \W$ where $\B^\ast$ is the subdisk of $\B(1,1)$ which is the isomorphic image of $\B(0,p^{\frac{p-2}{p-1}})$ via the map $s\mapsto (1+p)^s$, so
\begin{equation}\label{B*}
\B(0,p^{\frac{p-2}{p-1}})\cong\B^*\subseteq \B(1,1).
\end{equation}
In the notation introduced in Definition \ref{def:weightSp}, the character $\chi\langle\!\langle \ \rangle\!\rangle^k$ is represented by the point $(\chi,(1+p)^k)\in\W(\Q_p )$ which gets mapped to $(\chi,k)$ by the identification in \eqref{B*}; we see that the word ``coordinates'' comes from seeing $\W^*$ as $\varphi(Np)$-copies of $\B(0,p^{\frac{p-2}{p-1}})$. From now on we systematically write points in the weight space through their coordinates. In particular, for every fixed $\chi\in\dualD$, the weights of characters with Nebentypus $\chi$ will be regarded as points in $\B(0,p^{\frac{p-2}{p-1}})$ rather than in $\B(1,1)$.

The assumption $(N,p)=1$ allows us to consider the group $\mathrm{Hom}\big((\mathbb{Z}/p)^\times,\C^\times\big)$ as being a subgroup of $\dualD$. It thus makes sense, for each $0\leq j\leq p-2$, to interpret $\omega^j$ as an element of $\dualD$. The characters $x\mapsto x^k$, which are accessible with coordinates $(1,k)$, are then the elements of $\W^\ast(\Q_p)$ which belong to the $\omega^k$-th copy of $\B^*$.

\section{\texorpdfstring{$p$-adic families of modular forms}{}}\label{section:family_of_modular_forms}
We start with a classical eigencuspform $f\in S_{k_0}(\Gamma_1(Np),\varepsilon)$ of weight $k_0$, level $Np$ and Nebentypus $\varepsilon$, which we factor as a product $\varepsilon=\varepsilon_N\omega^{k_0-i}$ for some character $\varepsilon_N$ of conductor divisible by $N$ and some $0\leq i\leq p-1$. Our main references concerning $p$-adic modular forms and $p$-adic families thereof are~\cite[Part~B]{Col97} as well as~\cite{Gou88}, in particular Section II.3 \emph{ibid.}~for the definition of the $U_p$-operator.
\begin{defn} Let $f$ be a $p$-adic modular form of level $Np$ which is 
an eigenvector with respect to the $U_p$-operator. We define the slope of $f$ to be the $p$-adic valuation of the $U_p$-eigenvalue of $f$. It is a non-negative rational number.
\end{defn}
We assume that the slope of $f$ is $0\leq\alpha \leq k_0-1$. In \cite{Col97}, \cite{Col97c} and \cite{ColMaz98} Coleman and Coleman--Mazur have built a theory of families of $p$-adic modular forms of slope $\alpha$ interpolating $f$, a part of which is stated in Theorem \ref{Stevens} below. Let us introduce some notation. Given an element $k_0 \in \mathbb{Q}_p$, an integer $i$ and an element $r\in p^\mathbb{Q}$ with $r<p^\frac{p-2}{p-1}$, 
we denote by $\Xfr$ the affinoid subspace of the weight space defined as
\[
\Xf :=\{\varepsilon_N\omega^i\}\times\B[k_0,r]\subseteq \W^* . 
\]
Let us consider a finite extension $\K$ of $\Q_p$, which will play the role of the field of 
coefficients of motives associated to cuspforms $f$ in the given Coleman family. 
We note that $\K$ has nothing to do with the field of definition of motives  
associated to cuspforms $f$ in the given Coleman family, for which 
we can choose $\Q_p$. 
From now on, we denote by ${\OOO_{\Xfr}}_{/\K}$ (resp. ${\R_{\Xfr}}_{/\K}$) 
the extension of coefficients $\OOO_{\Xfr}  \otimes_{\Z_p} \mathcal{O}_\K$ 
(resp. $\R_{\Xfr} \otimes_{\Q_p}\K$) where $\mathcal{O}_\K$ 
is the ring of integers of $\K$. 
We have the following result thanks to \cite{Col97}: 
\begin{theorem} [\cite{Col97}]\label{Stevens} Suppose that $f$ is a classical normalized cuspidal eigenform, of weight $k_0$, level $\Gamma_1(Np)$, slope $\alpha<k_0-1$, Nebentypus $\varepsilon=\varepsilon_N\omega^{i-k_0}$ and which is new away from $p$. When $i=0$, suppose moreover that $a^2\neq \varepsilon_N(p)p^{k_0-1}$, where $a$ is the $U_p$-eigenvalue of $f$.
 
Then, there is a radius $r_0<p^\frac{p-2}{p-1}$ lying in $p^\mathbb{Q}$ 
and analytic functions $a_n\in{\OOO_{\Xf }}_{/\K}$, indexed by $n\in\N$, such that the following statements hold: 
\begin{enumerate}[label=\textup{(}\arabic*\textup{)}]
\item 
For all integers $k \in \Xf (\Q_p )$ satisfying 
$k> \alpha+1$, the series
\[
\displaystyle{\sum^\infty_{n= 1}} a_n(k)q^n \in \K [\![q]\!]
\]
coincides with the $q$-expansion of a classical normalized cuspidal eigenform of level $Np$, weight $k$, slope $\alpha$ and character $\varepsilon_N\omega^{i-k}$. \label{pt:thmColeman:Fourier}
\item 
The series $\displaystyle{\sum^\infty_{n= 1}} a_n(k_0 )q^n \in \K [\![q]\!]$ coincides with the $q$-expansion of $f$.\label{pt:thmColeman:centre}
\item The space $\Xf$ is $a_p$-small in the sense of \cite[(\textbf{5.2})]{Kis03}.\label{pt:thmColeman:apsmall}
\end{enumerate}
\end{theorem}
\begin{remark} At first glance, it might look better to write $a_n^{(r_0)}$ for the analytic functions appearing in the statement, since the radius $r_0$ is not uniquely associated to $f$ and these functions might \emph{a priori} depend on its choice. 
However, we will prove in Corollary \ref{a_l^r} below that this is not necessary.
\end{remark}
\begin{remark}
In this paper, the property that $\Xf$ be $a_p$-small will not be used. Nevertheless, point~\ref{pt:thmColeman:apsmall} of Theorem~\ref{Stevens} is essential when dealing with interpolation of crystalline periods along the Coleman family, adapting techniques from~\cite{Kis03}.
\end{remark}
\begin{proof} 
In \cite[pp.~465--467]{Col97} (and especially along the proof of Corollary B5.7.1 \emph{ibidem}), 
Coleman attached to the $f$ chosen above the space $\Xf$ with a radius $r_0 \in p^\mathbb{Q}$ small enough. 
In \cite[pp.~465--467]{Col97}, this space is denoted simply by $B$ and a crucial step is
to take a finite \'etale affinoid algebra $\Rf$ over $\R_{\Xf}$ whose associated affinoid space parametrizes families of $p'$-new forms 
of slope $\alpha$ in the sense of \cite[p.~467, Definition]{Col97}. 
As in \cite{Col97}, we write $X(\Rf)\to \Xf$ for the affinoid space associated to $\Rf$. 
 
Then the statements \ref{pt:thmColeman:Fourier}--\ref{pt:thmColeman:apsmall} are in \cite[Corollary~B5.7.1]{Col97} \emph{verbatim}, except for the condition that all $a_n$ be power-bounded and that all forms above be normalized. 
In Lemma B5.3 \emph{ibidem} it is shown that the Hecke eigenvalues of an overconvergent cuspidal eigenform are bounded by $1$ if the form is normalized. We thus deduce the result by observing that $a_1=1$, which follows from $a_1=T(1)=1$ by the construction given in Theorem B5.7 \emph{ibidem}. As discussed in \cite[\S 5.2]{Kis03} it is always possible to shrink a disk around a point in order to get a smaller one which is $a_p$-small, and we define this smaller radius as the constant $r_0$ of the statement.
\end{proof}
\begin{defn}\label{def:families} Let $f$ be a form as in Theorem \ref{Stevens} and let $0<r\leq r_0$ be smaller than or equal to the radius constructed \emph{ibid}. 
The collection of the rigid functions $\{a_n\}_{n\in \mathbb{N}}$ on $\Xfr$ is called the Coleman family of slope $\alpha$ and radius $r$ passing through the form $f$. 
The formal power series
\[
\sum^\infty_{n=1}a_n q^n\in {\OOO_{\Xfr}}_{/\K}[\![q]\!]
\]
is said to be the Fourier expansion of the Coleman family. For each $x\in \Xfr$, we denote by $f_x$ the overconvergent modular form whose expansion is $\sum a_n(x)q^n$. Sometimes, we also refer to the collection of all these forms as the Coleman family of slope $\alpha$ through the form $f$.
\end{defn}

Consider the subset $\Zfr$ of accessible weight-characters in $\Xfr$ defined as
\begin{equation}\label{equation:definitionZfr}
\Zfr =\Big(\{\varepsilon_N\omega^i\}\times\Z_{>\alpha+1}\Big)\cap\Xfr \subseteq \W^*(\Q_p ) .
\end{equation}
By an abuse of notation, we write $k$ for elements $(\varepsilon_N\omega^i,k)$ 
of $\Zfr$ and we denote the $p$-adic Deligne representation attached to $f_k$ in \cite{Del68} by $\rho_k$.

Consider now a form $f$ of weight $k_0$ satisfying the assumption of Theorem \ref{Stevens} and a radius $r_0$ as constructed \emph{ibid.} 

\begin{corollary}\label{a_l^r} 
Let $0<r<r_0$ be an element of $p^{\mathbb{Q}}$. For $n\geq 1$ the functions
\[
\mathrm{res}^{\Xf}_{\Xfr}(a_n ) \in {\OOO_{\Xfr}}_{/\K}
\]
are the Fourier expansion of a Coleman family of slope $\alpha$ and radius $r$ passing through the form $f$.
\end{corollary}
\begin{proof} This follows from the definitions, since for every $k\in \ZFfr (\Q_p )$ the series
\[
\sum^\infty_{n=1}\mathrm{res}^{\Xf}_{\Xfr}(a_n)(k)q^n
\]
is the $q$-expansion of a form of the required type.
\end{proof}
Thanks to the above corollary, we can unambiguously speak about the Fourier coefficients $a_n$  of a Coleman family 
without referring to the radius; observe also that 
 shrinking an $a_p$-small disk to one of smaller radius preserves $a_p$-smallness. We recall the following lemma:
\begin{lemma}\label{basic:lemma} 
Let $r\in p^{\mathbb{Q}}$ be a radius satisfying 
$r\leq r_0$. 
If the function $G \in { \R_{\Xfr}}_{/\K}$ vanishes on $\ZFf$, 
it is everywhere zero. 
\end{lemma} 
\begin{proof}  
The subset $\ZFfr \subseteq \Xfr (\Q_p )$ is Zariski-dense in $\Xfr$ thanks to 
the second assertion of Lemma \ref{basic:Zariski-dense}. Hence, 
the assertion is an immediate consequence of the first assertion of 
Lemma \ref{basic:Zariski-dense}. 
\end{proof}

\section{\texorpdfstring{$p$-adic family of Galois representations}{}}\label{sec:biggaloisrep}
The main result of this section is Theorem \ref{BigGaloisRep} which produces an \emph{integral} Galois representation with values in ${\OOO_\X}_{/\K}$ -- for 
a given Coleman family and a suitable affinoid $\X$ -- that specializes 
to the Deligne representations attached to the classical eigenforms which belong to 
the given Coleman family. 
The most important ingredient is to construct pseudo-representations associated 
to a given Coleman family: once we have a pseudo-representation over our affinoid algebras, a standard argument allows us to recover a Galois representation from this 
pseudo-representation. 
Since we are not aware of a reference on constructing pseudo-representations over affinoid algebras, we construct them here by using a formal structure of Coleman families. 
  
 Let us fix a Coleman family as in Theorem \ref{Stevens} and  
let $r_0 \in p^{\mathbb{Q}}$ be the radius constructed \emph{ibid.}
Let $K$ be a complete subfield of $\C$ and take an element $e_0 \in K$ such that 
$r_0 = \vert e_0 \vert$.  Define $\Ao_K$ to be the complete local ring
\begin{equation}\label{equation:A0}
\Ao_{K}=\mathcal{O}_{K}\Big[\!\!\Big[ \frac{T-k_0 }{e_0} \Big]\!\!\Big]=
\Big\{\sum^\infty_{i=0}c_i 
\left( \frac{T-k_0 }{e_0 }\right)^i \ \Big\vert \ c_i\in\mathcal{O}_{K}\Big\}.
\end{equation}
For any $r \in p^{\mathbb{Q}}$ with $r < r_0$, a power series $F$ in $\Ao_{K}$ converges on $B[k_0 ,r]_K \subsetneq B(k_0 ,r_0 )_K$ 
and its evaluations at each point $x \in B[k_0 ,r]_K$ satisfies $\vert F(x) \vert \leq 1$. 
We consider elements of $\Ao_K$ as functions on $\Xfr$, and this induces, by restriction, a ring homomorphism $\Ao_{K} 
\xrightarrow{\operatorname{res}} \OOO_{\Xfr _K}$. For each radius $r<r_0$, we choose $e_r \in \C$ such that 
$\vert e_r \vert = r$. By \cite[\S 6.1.5]{BosGunRem84}, there is an isomorphism
\[
\OOO_{\Xfr _{\C}}\cong 
 \Big\{\sum^\infty_{i=0}c_i \left( \frac{T-k_0}{e_r }\right)^i \in 
 \mathcal{O}_{\C}\Big[\!\!\Big[ \frac{T-k_0}{e_r} \Big]\!\!\Big]
 \ \Big\vert \ \lim_{i\to\infty}\vert c_i\vert =0\Big\}. 
\]
Similarly as above, the inclusion 
$\{\varepsilon_N\omega^i\}\times\B(k_0,r_0 )_K \subset \Xf_K$ induces a homomorphism $\OOO_{\Xf _K}\rightarrow \Ao_{K}$ 
where we have a presentation
\[
\OOO_{\Xf _K}\cong 
 \Big\{\sum^\infty_{i=0}c_i \left( \frac{T-k_0}{e_0 }\right)^i \in \mathcal{O}_{K}\Big[\!\!\Big[ \frac{T-k_0}{e_0} \Big]\!\!\Big]
 \ \Big\vert \ \lim_{i\to\infty}\vert c_i\vert =0\Big\}. 
\]
The composition
\begin{equation}\label{A_X to A to A_x}
\OOO_{\Xf _K}\longrightarrow \Ao_{K} \overset{\operatorname{res}}{\hookrightarrow} \OOO_{\Xfr _{K}}
\end{equation}
corresponds to the inclusions
\[
\B[k_0,r]_{\C} \subset \B(k_0,r_0)_K \subset \B[k_0,r_0]_K .
\]
In Figure~\ref{fig:radii}, there is a sketch of the radii that have occurred so far in our construction.

\begin{figure}
\xy
<22.8em,-7em>*+{\bullet}; 
<23.25em,-7.75em>*+{1}; 
<20.9em,-7em>*+{\bullet}; 
<21.45em,-7.75em>*+{r_0}; 
<18.03em,-7em>*+{\bullet} ; 
<18.53em,-7.75em>*+{r}; 
<-1em,-7em>*+{}; 
<25em,-7em>*+{} **@{-}; 
<-5em,-7em>*+{}; 
<29em,-7em>*+{} **@{--}; 
<11.75em,-7em>*+{\bullet}; 
<11.75em,-8em>*+{k_0}; 
*\xycircle(42.275,42.275){}
*\xycircle(35,35){}
*\xycircle(34.5,34.5){--}
*\xycircle(24.15,24.15){}
\endxy
\caption{}\label{fig:radii}
\end{figure}
When the radius $r \in p^{\mathbb{Q}}$ tends to $r_0$ from below, we have 
\[
\Ao_{K} \subset \Ao_{\C} = \varprojlim_{\begin{smallmatrix}
r\rightarrow r_0 \\ r<r_0
\end{smallmatrix}} \OOO_{\Xfr_{\C}} = 
\underset{r < r_0}{\bigcap } \OOO_{\Xfr_{\C}}. 
\]
From now on, we denote 
$\Ao_{\Q_p}$ by $\Ao$ when $K=\Q_p$ dropping the subscript $K$. We summarize the above observations
in the following proposition: 
\begin{proposition}\label{prop:Stevens2} 
Maintain the same assumptions and notations of Theorem \ref{Stevens}. There exists a finite extension $\K$ of $\Q_p$ such that, 
for every $n\in\N$, there is a unique function $A_n \in 
\AoK$ such that the image of $A_n$ via the restriction
$\AoK \hookrightarrow {\OOO_{\Xfr }}_{/\K}$ coincides with the function $a_n$ obtained in 
Theorem \ref{Stevens} on a smaller radius $r$. 
\end{proposition}
\begin{proof}
In Theorem \ref{Stevens}, we obtained a Coleman family over $\Xf$. 
By shrinking such a Coleman family to $\Xfr$ for any $r \in p^{\mathbb{Q}}$ with $r < r_0$ and by taking 
the limit $r\longrightarrow r_0$, we showed in Corollary~\ref{a_l^r} that there exists $A_n \in \AoK$ 
such that the image of $A_n$ via the inclusion 
${\AoK} \hookrightarrow {\OOO_{\Xfr }}_{/\K}$ coincides with $a_n $ obtained in 
Theorem \ref{Stevens}. 
This completes the proof of the proposition. 
\end{proof}

Based on the existence of the formal structure ${\AoK}$, we introduce the notion of $2$-dimensional pseudo-representation based on \cite{Wil88}. We remark that other types of pseudo-representations were later introduced by Taylor in \cite{Tay91} and, much more recently, by Chenevier in \cite{Che11} but for our purposes Wiles' approach seems to be the best suited. 

Given a topological group $G$, a topological ring $R$ in which $2$ is invertible, and three functions $A,D\colon G\longrightarrow R$ and 
$\Xi\colon G\times G\longrightarrow R$, we consider the following four properties given in \cite[Lemma~2.2.3]{Wil88}: 
\begin{enumerate}[label=\textup{(\Roman*)}]
\item $A$, $D$ and $\Xi$ are continuous functions.\label{pt:pseudorep:cont}
\item The relations 
\begin{align*}
& A(\sigma  \tau) = A (\sigma ) A (\tau) + \Xi (\sigma ,\tau ), \\  
& D(\sigma  \tau) = D (\sigma ) D (\tau) + \Xi (\tau ,\sigma ),\\
& \Xi (\sigma \tau , \rho \gamma ) 
= A (\sigma ) A (\gamma) \Xi (\tau ,\rho )
+ A (\gamma ) D (\tau) \Xi (\sigma ,\rho ) 
+ A (\sigma ) D (\rho)  \Xi (\tau ,\gamma ) 
+ D (\tau ) D (\rho) \Xi (\sigma ,\gamma )  
\end{align*}
hold for all elements $\sigma , \tau , \gamma , \rho \in G$. \label{pt:pseudorep:prd}
\item The relations $A(\boldsymbol{1})=D(\boldsymbol{1})=1$ and $\Xi (\sigma ,\boldsymbol{1})=\Xi (\boldsymbol{1} ,\tau )=0$ hold
for all elements $\sigma , \tau  \in G$ where $\boldsymbol{1}$ is the unit element of $G$. \label{pt:pseudorep:id}
\item The relation $\Xi (\sigma ,\tau) \Xi (\rho ,\eta) = \Xi (\sigma ,\eta) \Xi (\rho ,\tau )$ holds for all elements $\sigma , \tau  , \rho 
,\eta \in G$. \label{pt:pseudorep:Xi}
\end{enumerate}
We call a triple $(A,D,\Xi)$ satisfying the above conditions \ref{pt:pseudorep:cont}--\ref{pt:pseudorep:Xi} a pseudo-representation.
 
Given a continuous representation $\rho\colon G\to \operatorname{GL}_2(R)$, by fixing a basis of $R^2$ we can write
$$
\rho(\sigma)=\left(\begin{array}{cc}a(\sigma)&b(\sigma)\\c(\sigma)&d(\sigma)\end{array}\right)\qquad\text{for }\sigma\in G\
$$
and the triple $\pi_\rho=\bigl(A(\sigma )=a(\sigma ),D(\sigma )=d(\sigma ),\Xi(\sigma,\tau)=b(\sigma)c(\tau)\bigr)$ 
is easily checked to be a pseudo-representation; observe that attaching a pseudo-representation to a representation depends on the choice of a basis. Extending the notation introduced in Definition \ref{def:families}, and in the set-up introduced \emph{ibid}, we simply denote by $\pi_k$ the pseudo-representation $\pi_{\rho_k}$ attached to $\rho_k$ in some chosen basis. 
  
Let $\Zf$ be the set of points inside $\W^\ast(\Q_p)$ defined as the intersection
\begin{equation}\label{equation:definitionZf}
\Zf =\B(k_0,r_0)(\Q_p) \cap \ZFf
\end{equation}
where $\ZFf$ is the set defined in~\eqref{equation:definitionZfr}. According to the convention established at the end of Section~\ref{sec:affinoids}, weights in $\Zf$ will be identified with integers and denoted by a single coordinate.
\begin{proposition}\label{Pseudo-Rep} 
Let $S$ be the finite set of primes of $\mathbb{Q}$ 
consisting of the primes $\{\ell:\ell\mid N\}$, $p$ and $\infty$, and denote by $G_{\mathbb{Q},S}$ 
the Galois group of the maximal extension of $\Q$ unramified outside $S$. 
Then, there exists a continuous pseudo-representation 
\[
\pi = (A,D,\Xi)\colon G_{\mathbb{Q},S}\longrightarrow {\AoK}\]
interpolating the pseudo-representations $\pi_k$ attached to members of the Coleman family of slope $\alpha$ through $f$. 
In other words, for each $k\in \Zf$, the evaluation $\mathrm{ev}_k\circ\pi=(\mathrm{ev}_k\circ A,\mathrm{ev}_k\circ D,\mathrm{ev}_k\circ \Xi)$ coincides with 
the pseudo-representation $\pi_k$.
\end{proposition}
\begin{proof} 
The argument relying on pseudo-representations of rank two \`{a} la Wiles is more or less standard and the proof goes in a parallel manner 
as the proof given in the monograph~\cite[\S 7.5]{Hid93}, so we only give an outline. 
 
Let us choose a complex conjugation $c \in G_{\mathbb{Q},S}$. 
For each $k\in \Zf$, we fix a basis for the representation $\rho_k$ 
so that $c$ is represented by the matrix 
$
\begin{pmatrix} -1 & 0 \\ 
0 & 1 \\ 
\end{pmatrix}
$. 
Following \cite[Lemma~2.3.3]{Wil88},  
we define functions
\[
A_{(k)},D_{(k)}\colon G_{\mathbb{Q},S}\longrightarrow \K \\
\]
by the matrix representation 
\[
\rho_k(g)=\left(\begin{array}{cc}
A_{(k)}(g)&B_{(k)}(g)\\
C_{(k)}(g)&D_{(k)}(g)
\end{array}\right)\;.
\]
We also set $\Xi_{(k)}\colon G_{\mathbb{Q},S}\times G_{\mathbb{Q},S}\longrightarrow \K$ by $\Xi_{(k)}(\gamma_1,\gamma_2)=A_{(k)}(\gamma_1\gamma_2)-A_{(k)}(\gamma_1)A_{(k)}(\gamma_2)$. 
The functions $A_{(k)},D_{(k)},\Xi_k$ are continuous for every $k\in \Zf$ 
thanks to continuity of the Deligne representation $\rho_k$.
  
For any $k\in \Zf$, we denote by $P_k \subset {\AoK}$ the kernel of the evaluation 
map $ {\AoK} \longrightarrow \K$ at $k$. 
Let us denote the function $A_{(k)} + D_{(k)}$ by $\mathrm{Tr}_{P_k}$. 
Since $\Zf$ is a countable set, we choose a numbering 
\[
\Zf = \{ k_1 , k_2 , \ldots , k_s ,\ldots \} . 
\] 
For $k_1,k_2 \in  \Zf$, consider the map  
\begin{equation}\label{equation:glueing}
{\AoK} /P_{k_1} \oplus {\AoK}/P_{k_2} \longrightarrow  {\AoK} /(P_{k_1} +P_{k_2}) ,\ (x,y) \mapsto 
(x\text{ mod $P_{k_2}$}) - (y\text{ mod $P_{k_1}$} )
\end{equation}
whose kernel is isomorphic to ${\AoK} /(P_{k_1 } \cap P_{k_2} )$.
Let $\ell$ be a prime number outside $S$. 
We have $\mathrm{Tr}_{P_k}(\operatorname{Frob}_\ell)=a_\ell (k)$ for every $k\in \Zf$ 
and the Fourier coefficients $a_\ell (k_i)$ glue together when $k_i$ varies. 
Hence the values $\mathrm{Tr}_{P_k}(\operatorname{Frob}_\ell)$ glue together 
when $k$ varies, which is true for any prime number $\ell$ outside $S$. 
By the sequence \eqref{equation:glueing}, we have a continuous function 
$\mathrm{Tr}_{P_{k_1} \cap P_{k_2}} \colon 
G_{\mathbb{Q},S} \longrightarrow {\AoK} /(P_{k_1} \cap P_{k_2})$ 
whose value $\mathrm{Tr}_{P_{k_1} \cap P_{k_2}}(\operatorname{Frob}_\ell)$ 
is congruent to $\mathrm{Tr}_{P_{k_1}}(\operatorname{Frob}_\ell)$ 
(resp.~$\mathrm{Tr}_{P_{k_2}}(\operatorname{Frob}_\ell)$) mod $P_{k_1}$ 
(resp.~mod $P_{k_2}$) for every prime number $\ell$ outside $S$. 
Since the set of Frobenius elements $\{\mathrm{Frob}_{\ell}\}_{\ell\notin S}$ is dense in $G_{\mathbb{Q},S}$, it follows that the 
value $\mathrm{Tr}_{P_{k_1} \cap P_{k_2}}(\sigma )$ 
is congruent to $\mathrm{Tr}_{P_{k_1}}(\sigma )$ 
(resp.~$\mathrm{Tr}_{P_{k_2}}(\sigma )$) mod $P_{k_1}$ 
(resp.~mod $P_{k_2}$) for every $\sigma \in G_{\mathbb{Q},S}$. 
 
By an inductive argument, for each natural number $s$, 
we have a continuous function 
\[
\mathrm{Tr}_{P_{k_1} \cap P_{k_2} \cap \ldots \cap P_{k_s}} \colon 
G_{\mathbb{Q},S} \longrightarrow {\AoK} /(P_{k_1} \cap P_{k_2} \cap \ldots \cap P_{k_s})
\] 
and reducing it $\pmod{P_{k_i}}$ we recover the function 
$\mathrm{Tr}_{P_{k_i}}$. Note that ${\AoK} = \varprojlim_s {\AoK} /(P_{k_1} \cap P_{k_2} \cap \ldots \cap P_{k_s})$ 
since ${\AoK}$ is complete and local. We can thus define the continuous function 
$\mathrm{Tr}\colon 
G_{\mathbb{Q},S} \longrightarrow {\AoK} $ to be $ \varprojlim_s \mathrm{Tr}_{P_{k_1} \cap P_{k_2} \cap \ldots \cap P_{k_s}}$. 
 
We define the desired functions $A$ and $D$ 
as
\begin{equation*}
A (\sigma )= 
\frac{\mathrm{Tr} \big(\sigma \big)-\mathrm{Tr}\big(c\cdot\sigma \big)}{2} 
\qquad\text{ and }\qquad
D (\sigma )= 
\frac{\mathrm{Tr} \big(\sigma \big)+\mathrm{Tr}\big(c\cdot\sigma \big)}{2}.
\end{equation*} 
Since ${\AoK}$ is complete and local, 
a similar argument yields the construction of $\Xi$, 
by first evaluating the value $b(g_1) c(g_2)$  at pairs 
$(g_1,g_2)=(\mathrm{Frob}_{\ell_1},\mathrm{Frob}_{\ell_2})$ 
and then observing that the set of these pairs is dense in $G_{\mathbb{Q},S}\times G_{\mathbb{Q},S}$. We thus obtain a continuous function 
\begin{equation*}
\Xi \colon G_{\mathbb{Q},S}\times G_{\mathbb{Q},S}\longrightarrow {\AoK}
\end{equation*}
that recovers the function $\Xi_{(k)}$ by taking the reduction modulo $P_k$ 
of the function $\Xi$ for every $k\in \Zf$. 
 
Setting $\pi:=(A,D,\Xi)$, we need to check that the triple satisfies properties~\ref{pt:pseudorep:prd}--\ref{pt:pseudorep:Xi} above, using the fact that $\Zf$ is dense in the reduced $\X$. First, let us verify property~\ref{pt:pseudorep:Xi}, namely that for each $g_1,g_2,h_1,h_2\in G_{\mathbb{Q},S}$, the following holds:
\begin{equation}\label{Wiles:(IV)}
\Xi(g_1,g_2)\Xi(h_1,h_2)-\Xi(g_1,h_2)\Xi(h_1,g_2)=0\;.
\end{equation}
We need to check that the function 
$\Xi(g_1,g_2)\Xi(h_1,h_2)-\Xi(g_1,h_2)\Xi(h_1,g_2)$ vanishes identically on $\X$. 
Since the aforementioned property~\ref{pt:pseudorep:Xi} for the pseudo-representation $\pi_k=(A_{(k)},D_{(k)},\Xi_k)$,
we have 
\begin{multline*}
\big( \Xi(g_1,g_2)\Xi(h_1,h_2)-\Xi(g_1,h_2) \Xi(h_1,g_2)\big) (k) 
\\
=\Xi_k(g_1,g_2)\Xi_k(h_1,h_2)-\Xi_k(g_1,h_2)\Xi_k(h_1,g_2)
\end{multline*}
at each point $k$ in the dense subset $\Zf$. This proves the desired vanishing of 
\eqref{Wiles:(IV)} and 
the same argument shows the other properties~\ref{pt:pseudorep:prd} and~\ref{pt:pseudorep:id}.  
This completes the proof. 
\end{proof}
Theorem \ref{BigGaloisRep} below, which is the main result of this section, shows that the pseudo-representation 
that we have just constructed comes from a true representation. 
\begin{theorem}\label{BigGaloisRep} 
Under the same assumptions and notations 
as in Theorem~\ref{Stevens} and in Proposition~\ref{Pseudo-Rep}, there exists a free ${\AoK}$-module $\TT$ of rank two with a continuous $\GS$-action such that the representation 
\[
\boldsymbol{\rho}\colon \GS\longrightarrow\mathrm{Aut}_{{\AoK}}(\TT )
\]
satisfies $\pi_{\boldsymbol{\rho}}=\pi$ (in a suitable basis). In particular, $\boldsymbol{\rho}$ modulo $P_k$ is isomorphic to a lattice in the space $V_k$ of $\rho_k$ for all $k\in \Zf$.
\end{theorem}
The following proof mainly relies on \cite[Lemma~2.2.3]{Wil88}, see also \cite[Proposition~1, \S 7.5]{Hid93}.
\begin{proof} Start with a radius $r_0$ as in Theorem \ref{Stevens} and set $r=r_0\vert\varpi\vert$, where $\varpi$ is a uniformizer of $\K$. Let $\pi=(A,D,\Xi)$ be the ${\AoK}$-valued pseudo-representation constructed in Proposition \ref{Pseudo-Rep}. There exists a pair of elements $\sigma,\tau\in\GS$ such that $\Xi(\sigma,\tau)(k_0) \neq 0$. 
If not, the diagonal Galois representation $g\mapsto A(g)(k_0)\oplus D(g)(k_0)$ would have the same trace and determinant as the representation 
$\rho_k$, contradicting Ribet's result~\cite[Theorem~2.3]{Rib77} stating that $\rho_k$ be irreducible 
(see~\cite[Lemma~2.2.3]{Wil88} or~\cite[Proposition~1.1]{Hid89}). 
From now on, let us fix a pair $\sigma,\tau$ such that $\Xi(\sigma,\tau)(k_0)\in \K$ is of minimal valuation, say $\mu\in\mathbb{Q}$. 
The element $\Xi(\sigma,\tau)$ can be decomposed as $\Xi(\sigma,\tau)=p^\mu V^{-1}$ with $V\in({\AoK})^\times$. As in \cite[Proposition~1]{Hid93}, we check that the map
\[
\boldsymbol{\rho}\colon g\mapsto\left(\begin{array}{cc}
A(g) & \Xi(g,\tau)Vp^{-\mu}\\\\
\Xi(\sigma,g)&D(g)
\end{array}\right)
\]
is multiplicative, sends $\boldsymbol{1}\in \GS$ to $\mathbf{Id}_2\in M_2({\AoK} )$ and takes values in $M_2({\AoK} \otimes \mathbb{Q}_p)$. Hence we have a continuous group homomorphism $\boldsymbol{\rho}\colon \GS\to\operatorname{GL}_2({\AoK} \otimes \mathbb{Q}_p)$. 
In order to produce a finitely generated ${\AoK}$-submodule of ${\AoK} \otimes \mathbb{Q}_p$ which is $\boldsymbol{\rho}$-stable, we follow the proof of continuity of \cite[Proposition~1.1]{Hid89}. Namely, define $\mathfrak{J}\subseteq {\AoK} $ to be the ideal generated by all lower-left entries $\Xi(\sigma,g)$ for $g\in\GS$ and let $\TT '$ be the 
${\AoK}$-submodule of $({\AoK} )^{\oplus 2}\subset ({\AoK} \otimes \mathbb{Q}_p )^{\oplus 2}$ generated by all vectors $(x,y)$ with 
$x\in {\AoK} ,y\in\mathfrak{J}$. 
The ${\AoK}$-module $\TT '$ is finitely generated since 
${\AoK}$ is noetherian. Also, $\TT '$ is $p$-torsion free since $\TT '$ is contained in the module $({\AoK} \otimes \mathbb{Q}_p )^{\oplus 2}$ which 
has no non-zero $p$-torsion element. Moreover, 
given $^t(x,y)\in \TT '$ and $g\in \GS$, we have 
\[
\boldsymbol{\rho}(g)\cdot\left(\begin{array}{c}
x \\ y
\end{array}\right)=
\left(\begin{array}{c}
A(g)x+\Xi(g,\tau)Vp^{-\mu}y\\
\Xi(\sigma,g)x+D(g)y
\end{array}\right)
\in \TT '.
\]
Indeed, the definition of $\mathfrak{J}$ ensures $\Xi(\sigma,g)x\in\mathfrak{J}$, 
and the definition of $\mu$ guarantees that the element $\Xi(g,\tau)Vp^{-\mu}$ belongs to ${\AoK}$. 
It follows that $\TT'$ is Galois stable. Finally, we have the equality
$\TT' \otimes\Q_p=(\Ao \otimes \mathbb{Q}_p )^{\oplus 2}$ because $\mathfrak{J}$ contains the element $\Xi(\sigma,\tau)=p^{-\mu}V$ which becomes a unit after inverting $p$.
Define $\TT $ to be the double dual $\mathrm{Hom}_{\AoK} (\mathrm{Hom}_{\AoK} (\TT' ,{\AoK}), {\AoK} )$. We have a $\GS$-equivariant ${\AoK}$-linear injection 
$\TT' \hookrightarrow \TT$ with finite cokernel. Since $\TT$ is a finitely generated reflexive module over a regular local ring ${\AoK}$ of Krull dimension 
two, $\TT$ is a free ${\AoK}$-module of finite rank. Since $\TT \otimes \Q_p$ is free of rank two over ${\AoK} \otimes \Q_p$, the rank of $\TT$ over ${\AoK}$ is two.
 
As for the interpolation property, recall that an irreducible representation of $\GS$ with values in a finite extension of $\Q_p$ is uniquely determined by its trace and determinant. 
Thus $\TT/P_k \TT$ is isomorphic to a lattice inside $V_k$.
\end{proof}
Note that $\iota \colon {\AoK} \hookrightarrow { \OOO_{\Xfr} }_{/\K}$ is continuous. 
In fact, ${\AoK}$ is a local ring with maximal ideal 
$(\frac{T-k_0}{e_0} ,\varpi)$ and ${\OOO_{\Xfr}}_{/\K}$ is endowed with the $\varpi$-adic topology, where $\varpi$ denotes as before a uniformizer of $\K$. To show that $\iota$ is continuous, we need to show that 
$\iota^{-1}\Big(\varpi {\OOO_{\Xfr}}_{/\K}\Big)$ contains the maximal ideal 
$(\frac{T-k_0 }{e_0},\varpi)$ of ${\AoK}$. 
In fact, this implies that $\iota^{-1}(\varpi^n {\OOO_{\Xfr}}_{/\K})$ contains 
$(\frac{T-k_0}{e_0} ,\varpi)^n$ for every $n$. 
The uniformizer $\varpi$ is clearly contained in 
$\iota^{-1}(\varpi { \OOO_{\Xfr}}_{/\K})$ and $\frac{T-k_0}{e_0}$ is also 
contained in $\iota^{-1}(\varpi {\OOO_{\Xfr}}_{/\K})$ since we have $\frac{T-k_0}{e_0} = 
\varpi \frac{T-k_0}{e_0 \varpi } \in \varpi {\OOO_{\Xfr}}_{/\K}$. 
By extending the coefficients of the result obtained over ${\AoK}$ in the Theorem~\ref{BigGaloisRep}  to ${\OOO_{\Xfr}}_{/\K}$, 
we obtain the following corollary:
\begin{corollary}\label{cor:BigGaloisRep}
Under the same assumptions and notations as in
Theorem~\ref{BigGaloisRep} , 
there exists a free ${\OOO_{\Xfr}}_{/\K}$-module $\T$ of rank two with a continuous $\GS$-action such that the representation
\[
\boldsymbol{\rho}\colon \GS\longrightarrow 
\mathrm{Aut}_{{\OOO_{\Xfr}}_{/\K}}(\T) 
\]
satisfies $\pi_{\boldsymbol{\rho}}=\pi$ $($in a suitable basis$)$. 
In particular, $\boldsymbol{\rho}$ modulo $\mathfrak{m}_k$ is isomorphic to $\rho_k$ for all $k\in \Zf$ where $\mathfrak{m}_k$ is the unique maximal ideal of 
${\OOO_{\Xfr}}_{/\K}$ such that $\mathfrak{m}_k \cap {\AoK} =P_k$.
\end{corollary}

\section{The fine Selmer group and Galois deformations}\label{sec:deformations}

We stick to the notation introduced in Theorem~\ref{Stevens} and in Definition~\ref{def:families}. Let $F$ be a finite extension of $\Q$ and fix $k\in \Zf$ (see~\eqref{equation:definitionZf}). The fine Selmer group of $f_k$ over $F$ is defined by 
\begin{equation*}
	\Sel(F,V_{f_k}/T_{f_k})=\Ker\Big(H^1(\Q_S/F,V_{f_k}/T_{f_k}) \rightarrow \oplus_{v \in S}H^1(F_v,V_{f_k}/T_{f_k})\Big),
\end{equation*}
where $T_{f_k}$ is the lattice of the representation $V_{f_k}$ of $\rho_k$ from Theorem \ref{BigGaloisRep} .
The fine Selmer over the extension $\Q(\mu_{p^\infty})$ is defined by 
\begin{equation}\label{eq:fine_selmer_def}
	\Sel(\Q(\mu_{p^\infty}), V_{f_k}/T_{f_k}) = \varinjlim \Sel(F,V_{f_k}/T_{f_k}),
\end{equation}
where the limit is over all finite extensions of $\Q$ contained in $\Q(\mu_{p^\infty})$.
Set $(T_{f_k})^\star(1)=\Hom_{\cts}(T_{f_k} , \OO)\otimes_{\Z_p} \Z_p(1)$ and let $\Y(\Q(\mu_{p^\infty}),(T_{f_k})^\star(1))$ be the Pontryagin dual of the group $\Sel(\Q(\mu_{p^\infty}), V_{f_k}/T_{f_k})$. Since $p$ is odd,
\[
\Gal(\Q(\mu_{p^\infty})/\Q) \cong \Delta \times (1+p\Z_p)
\]
where $\Delta$ is a finite group of order $p-1$. For a character $\eta$ of $\Delta$, by a result of Kato \cite[Theorem~12.4(1)]{Kato}, we know that the $\eta$-isotypic component  $\big(\Y(\Q(\mu_{p^\infty}),(T_{f_k})^\star(1)\big)^\eta$ is torsion over the Iwasawa algebra $\Z_p[[1+p\Z_p]]$. For simplicity we will restrict ourselves to the case $\Q(\mu_{p^\infty})$; one could have also worked with the dual fine Selmer group over the cyclotomic extension of any abelian number field where Kato's result is true.

\begin{remark}
	
For supersingular elliptic curves,  the dual fine Selmer group is  very close to the torsion subgroup of the whole classical dual Selmer group as suggested by the result~\cite[Corollary~2.5]{Wingberg} of Wingberg, saying that given a supersingular elliptic curve $E/F$ and assuming the finiteness of $\Sha(E/F_n)_{p^\infty}$ for all $n$, the torsion part of the dual Selmer group of $E$ over $F_\cyc$ is pseudo-isomorphic to the  dual fine Selmer group of $E$ over $F_\cyc$. Therefore, philosophically, under appropriate hypotheses, the dual fine Selmer group is not far from the torsion part of the dual classical Selmer group. Under certain hypotheses, another evidence in this direction was also provided by Billot over the  $\Z_p^2$-extension $F_\infty=E[p^\infty]$ of $F$, when $E/F$ is an elliptic curve with complex multiplication by an imaginary quadratic field $K$ in $F$ that has good supersingular reduction for primes above $p$ (see \cite[Theorem~3.23]{Billot}). In \cite[Theorem 1.1]{Matar}, Matar gave a Galois-theoretic proof of Wingberg's result  with a slightly different hypothesis. Wingberg's result has been generalized by Lei and Lim to modular forms which are non-ordinary at $p$ under appropriate hypotheses (see \cite[Theorem 3.4]{Lei-Lim22}). 
\end{remark}

Let $\Lcyc:=\Z_p[[\Gal(\Q(\mu_{p^\infty})/\Q)]]$ be the Iwasawa algebra  of the group $\Gal(\Q(\mu_{p^\infty})/\Q)$. Let $\Lcyc({\tchi})$ be a free $\Lcyc$-module of rank one on which the absolute Galois group $G_\Q$ of $\Q$ acts via the  cyclotomic character
\[
{\tchi}\colon G_\Q \twoheadrightarrow \Gal(\Q(\mu_{p^\infty})/\Q) \hookrightarrow \Lcyc^\times.
\]
The following definition is adapted from Greenberg (see~\cite{Greenberg1} and~\cite{Greenberg2}).
\begin{defn}
Write $\tL:={\AoK} \widehat{\otimes}_{\Z_p}\Lcyc$, and let $\TT$ be as in Theorem~\ref{BigGaloisRep}. The cyclotomic deformation of $\TT$ is defined as
\[
\tT:= \TT \widehat{\otimes}_{\Z_p} \Lcyc({\tchi})
\]
regarded as a module over ${\tL}$. Set $\tD:=\tT \otimes_{{\tL}} \widehat{{\tL}}$ where $\widehat{{\tL}}=\Hom_{\mathcal{O}_\mathcal{K}}({\tL},\mathcal{K}/\mathcal{O}_\mathcal{K})$ is the Pontryagin dual of ${\tL}$. We endow $\tD$ with the discrete topology.
	
Finally, set $(\tT)^\star=\Hom_{\tL}(\tT, \tL)$ to be the $\tL$-linear  dual of $\tT$, and let $\TT^\star=\Hom_{{\AoK}}(\TT, {\AoK})$ be the ${\AoK}$-linear dual of $\TT$.
\end{defn}

Note that 
\begin{align*}
\Hom_{\OK}(\tD, K/\OK(1)) &\cong \Hom_{\OK}(\tT \otimes_{\tL} \widehat{\tL},\mathcal{K}/\OK(1)) \\
&\cong \Hom_{\tL}(\tT, \Hom_{\OK}(\widehat{\tL},K/\OK)) \otimes_{\Z_p} \Z_p(1)\\
&\cong \Hom_{\tL}(\tT,\widehat{\widehat{\tL}})(1)\\
&\cong \Hom_{\tL}(\tT, \tL)(1) = (\tT)^\star(1).
\end{align*}
Similarly, one can define $$\D=\TT \otimes_{{\AoK}} \widehat{{\AoK}}$$ where $\widehat{{\AoK}}$ is the Pontryagin dual of ${\AoK}$. 

Suppose $\phi\colon {\AoK} \twoheadrightarrow \OK$ is any continuous surjection, set $P_\phi=\Ker(\phi)$ and $T_{\phi}=\TT/P_\phi \TT $. Define
\[
\D_\phi:=\TT_\phi \otimes_{\OK} \widehat{\OK}.
\]
By Greenberg \cite[p.~336]{Greenberg1}, we have $\D[P_\phi]\cong\D_\phi$, where $\D[P_\phi]$ denotes the submodule of $\D$ annihilated by $P_\phi$.

Similarly, by \emph{loc.~cit}, if ${\tphi}\colon\tL \twoheadrightarrow \OK$ is a continuous surjection and $P_{{\tphi}}$ denotes $\Ker({\tphi})$, let  $\TT_{{\tphi}}:= \tT/P_{{\tphi}}\tT$, which is a free $\OK$-module of rank two. Then,
\[
\tD[P_{{\tphi}}]\cong \tD_{{\tphi}}:=\TT_{{\tphi}} \otimes_{\OK} \widehat{\OK}.
\]

Following \cite[p.~336]{Greenberg1}, we define the fine Selmer group of $\tD$ as 
\[
\mathcal{R}(\Q,\tD) = \Ker\Big(H^1(\Q_S/\Q, \tD) \rightarrow \oplus_{v \in S}H^1(\Q_v, \tD)\Big).
\]
 Finally, write $\Y\big(\Q, (\tT)^\star(1)\big)$ for the Pontryagin dual $\Hom_{\OO}(\Sel(\Q,\tD),\K/\OO)$ of $\Sel(\Q,\tD)$.

\section{Specialization at arithmetic primes}\label{sub:specialization}
We continue to work in the setting introduced in \S\S\ref{section:family_of_modular_forms}--\ref{sec:biggaloisrep}, and we fix a weight $k\in\Zf$. Recall from \S\ref{sec:biggaloisrep} that $P_k$ is a  height-one prime ideal of ${\AoK}$. Let $J_k={P_k\tL}$ be the corresponding height-one prime ideal of $\tL$. We have a map 
\begin{equation}\label{eqn:spe}
\frac{\DFS}{J_k\DFS} \overset{s_k}{\longrightarrow} \mathcal{Y}\Big(\Q, \frac{(\tT)^\star(1)}{J_k(\tT)^\star(1)}\Big)
\end{equation}
which arises from taking Pontryagin dual of the map 
\begin{equation}
\Sel(\Q,\tD[J_k]) \overset{\widehat{s_k}}{\longrightarrow} \Sel(\Q,\tD)[J_k].
\end{equation}

Given any $v \in S$, let $G_{\infty,v}$ be the Galois group $\Gal(\overline{\Q_v}/\Q_{\infty,v})$ where $\Q_{\infty,v} =\Q(\mu_{p^\infty})_v$.
\begin{lemma}\label{eq:descend}
For all $v \in S$, we have
\begin{equation*}
H_0\big(G_{\Q_v},(\tT)^\star\big)\cong H_0(G_{\infty,v},\TT^\star)^{\iota}
\end{equation*}
as ${\AoK}$-modules, where the exponent $\iota$ denotes that the Galois group acts by the involution $g \mapsto g^{-1}$.
	\end{lemma}
\begin{proof}

	The Pontryagin duals of $(\tT)^\star$ and $\TT^\star$ are $\tD$ and $\D$, respectively. So the Pontryagin duals of $(\tT^\star)_{G_{\Q_v}}$ and $(\TT^\star)^{\iota}_{G_{\infty,v}}$ are $\tD^{G_{\Q_v}}$ and $(\D^{G_{\infty,v}})^\iota$. Therefore, we have to show that 
	\begin{align}\label{eq: mithu}
	\tD^{G_{\Q_v}} \cong (\D^{G_{\infty,v}})^\iota.
	\end{align}
	For this we will follow the line of arguments presented in \cite[p.~214-215]{Greenberg2} (see also \cite[Remark~5.9.2]{Greenberg1}).
	Note that 
	\begin{align*}
	\Hom_{\OK}(\tL, \mathcal{K}/\OK) &= \Hom_{\OK}({\AoK}\widehat{\otimes}_{\Z_p} \Lcyc, \mathcal{K}/\OK).
	\end{align*}
Since we are considering continuous maps and the topology on $\mathcal{K}/\OK$ is discrete, we have 
\[
\Hom_{\OK}({\AoK} \widehat{\otimes}_{\Z_p} \Lcyc, \mathcal{K}/\OK)=\Hom_{\OK}({\AoK} {\otimes}_{\Z_p} \Lcyc, \mathcal{K}/\OK).
\]
Since the Hom functor and the tensor product form an adjoint pair, we deduce
\[
\Hom_{\OK}({\AoK} {\otimes}_{\Z_p} \Lcyc, \mathcal{K}/\OK)=\Hom_{\Z_p}(\Lcyc, \widehat{{\AoK}}).
\]
Therefore
	\begin{align*}
	\tD = \Big(\TT \otimes_{\Z_p} \Lcyc({\tchi})\Big) \otimes_{\tL} \Hom_{\Z_p}(\Lcyc, \widehat{{\AoK}}).
	\end{align*}
	Now $\tT=\TT \otimes_{\Z_p} \Lcyc({\tchi})$ is a free, module of rank $2$ over $\tL$ and hence 
	\begin{align*}
	\tD &\cong \Hom_{\Z_p}\big(\Lcyc, \TT \otimes_{{\AoK}} \widehat{{\AoK}}\big)({\tchi})\\
	&\cong \Hom_{\Z_p}\big(\Lcyc({\tchi}^{-1}),\D\big).
	\end{align*}
We deduce that
	\[
	\tD^{G_{\infty,v}}\cong \Hom_{G_{\infty,v}}\big(\Lcyc({\tchi}^{-1}),\D\big) \cong \Hom_{\Z_p}\big(\Lcyc({\tchi}^{-1}), \D^{G_{\infty,v}}\big).
	\]
	Writing $\Gamma_v$ for the local cyclotomic Galois group $\Gal(\Q_{\infty,v}/\Q_v)=G_{\Q_v}/G_{\infty,v}$,
	this implies that 
	\begin{align*}
	\tD^{G_{\Q_v}} &\cong \Hom_{\Gamma_v}\big(\Lcyc({\tchi}^{-1}), \D^{G_{\infty,v}}  \big),\\
	&\cong \Hom_{\Gamma_v}\big(\Lcyc({\tchi}), (\D^{G_{\infty,v}}  )^\iota\big).
	\end{align*}
	The argument that follows is identical to that in~\cite[pp.~214--215]{Greenberg2}. One can identify $\Hom_{\Gamma_v}\big(\Lcyc({\tchi}), (\D^{G_{\infty,v}}  )^\iota\big)$ with $\Hom_{\Lcyc}\big(\Lcyc, (\D^{G_{\infty,v}}  )^\iota\big)$ considering $\Lcyc({\tchi})$ and $(\D^{G_{\infty,v}}  )^\iota$ as ${\Lcyc}$-modules. But any element $f \in \Hom_{\Lcyc}\big(\Lcyc, (\D^{G_{\infty,v}}  )^\iota\big)$ is determined completely by $f(1)$ and hence there is an isomorphism $\Hom_{\Lcyc}\big(\Lcyc, (\D^{G_{\infty,v}}  )^\iota\big) \cong (\D^{G_{\infty,v}}  )^\iota$ as  $\Lcyc$-modules. This completes the proof of \eqref{eq: mithu}.
\end{proof}
Let $M$ be the maximal ideal of ${\AoK}$, and set $\resfield=\AoK/M$. For the rest of the article,  we will need the following hypothesis.
\begin{hyp}\label{hyp}
	The residual $\resfield$-representation $\boldsymbol{\rho}/M$ is irreducible when restricted to $G_{\Q_p}$, where $\boldsymbol{\rho}$ is the Galois representation constructed in Theorem \ref{BigGaloisRep}.
\end{hyp}

\begin{theorem}\label{thm:specialization}
 Under Hypothesis \ref{hyp}, the map $s_k$ in \eqref{eqn:spe} is surjective and $\Ker(s_k)$ is a finitely generated $\OK$-module for all $k \in {\Zf}$. Furthermore, $\Ker(s_k)$ is finite for all but finitely many $k$.
\end{theorem}
\begin{proof}
	We have the following commutative diagram with exact rows.
		\begin{equation}\label{fundamental_diagram}
	\begin{tikzcd}
	0 \arrow[r] & \Sel(\Q,\tD[J_k]) \arrow[d, "\widehat{s_k}"] \arrow[r ] & H^1(\Q_S/\Q, \tD[J_k]) \arrow[d, "g_k"] \arrow[r] & \displaystyle{\bigoplus_{v \in S}H^1(\Q_v, \tD[J_k]) \arrow[d, "h_k"]}  \\
	0 \arrow[r] & \Sel(\Q,\tD)[J_k]\arrow[r] & H^1(\Q_S/\Q, \tD)[J_k]  \arrow[r] & \displaystyle{\bigoplus_{v \in S}H^1(\Q_v, \tD) [J_k]}
	\end{tikzcd}
	\end{equation}
Observe that $\D[M]$ (the submodule of $\D$ annihilated by $M$) is a finite dimensional $\resfield$-representation and, by Hypothesis~\ref{hyp},
\[
H^0(\Q_S/\Q, \D[M])= H^0(\Q_S/\Q, \D)[M]=0.
\]
Note that $\Lcyc({\tchi})$ is isomorphic to $p-1$ copies of the Iwasawa algebra in one variable over $\Z_p$. Hence, the residual $\resfield$-representation $\tT/M\tT$ associated to $\tT$ is isomorphic to a direct sum of $p-1$ copies of the irreducible residual representation of $\TT/M\TT$. Therefore, any $\resfield$-irreducible subquotient of the residual representation of  $\tT/M\tT$ is  isomorphic to the irreducible residual representation $\TT/M\TT$. 
Applying \cite[Proposition~3.4]{Greenberg1}, we deduce that the map $g_k$ is an isomorphism and hence $\widehat{s_k}$ is injective (that is, $s_k$ is surjective). In the following, we will analyse $\widehat{\Ker(h_k)}$ and show that $\widehat{\Ker(h_k)}$ is a finitely generated $\OK$-module for all $k$, and is finite for all but finitely many $k$. By the snake lemma, the same conclusion will hold for $\widehat{\Coker(\widehat{s_k})}\cong\Ker(s_k)$.
We remark that the Pontryagin dual of $\Ker(h_k)$ is isomorphic to $\oplus_{v \in S}(\tT^\star)_{G_{\Q_v}}[J_k]$.

By Lemma~\ref{eq:descend}, we obtain
\begin{equation*}
\bigoplus_{v \in S}(\tT^\star)_{G_{\Q_v}}[J_k]= \bigoplus_{v \in S}(\TT^\star)^{\iota}_{G_\infty,v}[P_k].
\end{equation*}
If $P_k$ does not divide the characteristic ideal of the torsion subgroup of $\oplus_{v \in S}(\TT^\star)^{\iota}_{G_\infty,v}$ as an ${\AoK}$-module, then we obtain that $\oplus_{v \in S}(\TT^\star)^{\iota}_{G_\infty,v}[P_k]$ is finite (see \cite[Lemma~4.1]{Ochiai1} for a generalization of this to  modules over Iwasawa algebras of several variables). Note also that, for all $k$, $\oplus_{v \in S}(\TT^\star)^{\iota}_{G_\infty,v}[P_k]$ is a finitely generated ${\AoK}/P_k$-module, and thus finitely generated over $\OO$ because $\AoK/P_k$ is finite over $\OO$. This shows that $\widehat{\Coker(\widehat{s_k})}$ is a finitely generated $\OK$-module for all~$k$, and it is finite for all but finitely many~$k$.

Let us denote
\[
{\tCT}_{P_k}:=(\TT/P_k\TT){\otimes}_{\Z_p} \Lcyc({\tchi})
\]
which is the cyclotomic deformation of $\TT/P_k\TT \cong T_{f_k}$. 
By \cite[p.~357]{Greenberg1},
\[
\tD[J_k] \cong {\tCT}_{P_k} \otimes_{\Lcyc} \widehat{\Lcyc}.
\] 
It follows from classical results of Greenberg (see, for example, \cite[(2), p.~342]{Greenberg1}) that 
\begin{equation}\label{eq:biswarup}
\Sel(\Q, {\tCT}_{P_k} \otimes_{\Lcyc} \widehat{\Lcyc}) \cong \Sel(\Q(\mu_{p^\infty}),V_{f_k}/T_{f_k})^\iota,
\end{equation}
where $\Sel(\Q(\mu_{p^\infty}),V_{f_k}/T_{f_k})$ is the  fine Selmer group defined in \eqref{eq:fine_selmer_def}.
Note that Greenberg's result is over the cyclotomic extension.  But our $\Lcyc$ is isomorphic to $\Z_p[[\Z_p^\times]]$ and hence we recover the fine Selmer group over $\Q(\mu_{p^\infty})$.  See also  \cite[Proposition~3]{Somnath} where Jha and Sujatha prove \eqref{eq:biswarup} in the ordinary case. The proof remains essentially the same in our case. Therefore,
\[
\Sel(\Q(\mu_{p^\infty}),V_{f_k}/T_{f_k})^\iota \overset{\widehat{s_k}}{\longrightarrow} \Sel(\Q,\tD)[J_k]
\]
is injective, $\widehat{\Coker(\widehat{s_k})}$ is a finitely generated $\OK$-module for all $k$, and is finite for all but finitely many~$k$. Dually, the map
\begin{equation}\label{eq:imp}
\frac{\DFS}{J_k\DFS} \overset{s_k}{\longrightarrow} \Y(\Q(\mu_{p^\infty}),T_{f_k}^\star(1))^\iota
\end{equation}
is surjective and $\Ker(s_k)$ is a finitely generated $\OK$-module for all~$k$. Furthermore, $\Ker(s_k)$ is finite for all but finitely many~$k$. This completes the proof of Theorem~\ref{thm:specialization}.
\end{proof}
Recall $\Z_p^\times \cong \Delta \times (1+p\Z_p)$ where $\Delta$ is a finite group of order $p-1$.
For a Dirichlet character $\eta\colon \Delta \rightarrow \Z_p^\times$, let $e_\eta$ be the idempotent attached to $\eta$. Given a $\Z_p[[\Z_p^\times]]$-module $M$ define its $\eta$-isotypic component to be
\[
M^\eta=e_\eta\cdot M,
\]
considered as a $\Z_p[[1+p\Z_p]]$-module. As in \cite[p.~3]{LeiPonsinet20}, we say that  $M$ has rank $r$ as a $\Z_p[[\Z_p^\times]]$-module if $M^\eta$ has rank $r$ as a $\Z_p[[1+p\Z_p]]$-module for all the characters $\eta$ of $ \Delta$. Hence, $M$ is torsion as a $\Z_p[[\Z_p^\times]]$-module if $M^\eta$ is torsion as a $\Z_p[[1+p\Z_p]]$-module for every character $\eta$.
\begin{proposition}\label{prop:torsion}
For every character $\eta$ of $\Delta$, the component $\DFS^\eta$ of the dual fine Selmer group is $\tL^\eta$-torsion.
\end{proposition}
\begin{proof}
Fix an $\eta$ as in the statement. By~\eqref{eq:imp}, choose $k$ such that the map
\[
\frac{\DFS}{J_k\DFS} \xrightarrow{s_k} \Y(\Q(\mu_{p^\infty}),T_{f_k}^\star(1))^\iota
\]
is surjective and the kernel of $s_k$ is finite. 
Using the Poitou--Tate exact sequence (see for example \cite[\S A.3.1]{perrinriou95}) and a deep result of Kato  \cite[Theorem~12.4]{Kato}, we know that $\Y(\Q(\mu_{p^\infty}),T_{f_k}^\star(1)) $ is torsion as a $\Lcyc$-module. Now suppose $\DFS^\eta$ is not $\tL^\eta$-torsion. Then by the structure theorem $\DFS^\eta$ can be decomposed into an $\tL^\eta$-free part of rank at least 1 and an $\tL^\eta$-torsion part. 
Recall that  ${J_k=P_k\tL}$ is a height one prime ideal of $\tL$, where $P_k$ is a  height one prime ideal of ${\AoK}$.
Hence the  $\tL^\eta$-free part of $\DFS^\eta$ modulo $J_k$ will contain a free $\Lcyc^\eta$ module of rank at least 1. But this is impossible as the kernel of $s_k$ is finite and  $\Y(\Q(\mu_{p^\infty}),T_{f_k}^\star(1))^\eta$ is $\Lcyc^\eta$-torsion.
\end{proof}
\section{\texorpdfstring{Variations of Iwasawa $\mu$ and $\lambda$-invariants}{}}\label{sub:maisection}
As before, in this section, we work under Hypothesis \ref{hyp}, and in the setting of \S\S\ref{section:family_of_modular_forms}--\ref{sec:biggaloisrep}.
From Proposition \ref{prop:torsion} and from the discussion before Theorem~\ref{introtheorem:Ray}, we know that for every character $\eta$ of $\Delta$, the component $\DFS^\eta$ of the fine Selmer group is $\tL^\eta$-torsion and the specialization $\Y(\Q(\mu_{p^\infty}),T_{f_k}^\star(1))^{\eta}$ is also $\Lcyc^\eta$-torsion for all points $k\in {\Zf}$.

\begin{theorem}\label{thm:main_mu}
For every character $\eta$ of $\Delta$, the following are equivalent:
\begin{enumerate}[label=\textup{(}\arabic*\textup{)}]
	\item The  fine Selmer group $\DFS^{\eta}$ is a finitely generated ${\AoK}$-module.\label{pt:thm_main_mu:fg}
	\item The classical dual fine Selmer group $\Y(\Q(\mu_{p^\infty}),T_{f_k}^\star(1))^{\eta}$ is a finitely generated $\OK$-module for all points $k\in {\Zf}.$ That is, the $\mu$-invariants of  $\Y(\Q(\mu_{p^\infty}),T_{f_k}^\star(1))^{\eta}$ is zero for all points $k\in {\Zf}$.\label{pt:thm_main_mu:zero_all}
	\item Part~\ref{pt:thm_main_mu:zero_all} above holds for some  $k$. That is, there exists $k$ such that the dual fine Selmer group $\Y(\Q(\mu_{p^\infty}),T_{f_k}^\star(1))^{\eta}$ is a finitely generated $\OK$-module for that $k$ (equivalently, its $\mu$-invariant is zero).\label{pt:thm_main_mu:zero_somes}
\end{enumerate}
\end{theorem}
\begin{proof}
Fix a character $\eta$ as in the statement. 

We first show that~\ref{pt:thm_main_mu:zero_all} and~\ref{pt:thm_main_mu:zero_somes} are equivalent. Fix a weight $k\in\Zf$ and let $A_{f_k}=V_{f_k}/T_{f_k}$. As before, let $\varpi$ be a uniformizer of $\OK$.  We have the following Kummer sequence associated to multiplication by $\varpi$:
\[
0 \rightarrow A_{f_k}[\varpi] \rightarrow A_{f_k} \xrightarrow{\varpi} A_{f_k} \rightarrow 0.
\]
 This gives us the following commutative diagram with exact rows.
 \[
 {\tiny 
 \begin{tikzcd}
 0 \arrow[r] & \big(\frac{H^0(\Q_S/\Q(\mu_{p^\infty}),A_{f_k})}{\varpi}\big)^\eta \arrow[d, "\alpha"] \arrow[r] & H^1\big(\Q_S/\Q(\mu_{p^\infty}), A_{f_k}[\varpi]\big)^\eta \arrow[d, "\beta"] \arrow[r] & \Big(H^1\big(\Q_S/\Q(\mu_{p^\infty}), A_{f_k}\big)[\varpi]\Big)^\eta  \arrow[d, "\gamma"] \arrow[r] & 0 \\
 0 \arrow[r] & \big(\frac{\oplus_{v \in S}H^0(\Q(\mu_{p^\infty})_v,A_{f_k})}{\varpi}\big)^\eta \arrow[r] & \oplus_{v \in S}H^1\big(\Q(\mu_{p^\infty})_v, A_{f_k}[\varpi]\big)^\eta \arrow[r] & \Big(\oplus_{v \in S}H^1\big(\Q(\mu_{p^\infty})_v, A_{f_k}\big)[\varpi]\Big)^\eta   \ar[r] & 0
 \end{tikzcd}
}
 \]
\noindent Each of the first terms in the two horizontal exact sequences is finite. By definition, the kernel of $\beta$ is $\Sel\big(\Q(\mu_{p^\infty}), A_{f_k}[\varpi]\big)^\eta$ which is the $\eta$-component of the  fine  Selmer group corresponding to the module $A_{f_k}[\varpi]$ over the extension $\Q(\mu_{p^\infty})$. Clearly, the kernel of $\gamma$ is $\Big(\Sel\big(\Q(\mu_{p^\infty}), A_{f_k}\big)[\varpi]\Big)^\eta$. Therefore, by the snake lemma $ \Big(\Sel\big(\Q(\mu_{p^\infty}), A_{f_k}\big)[\varpi]\Big)^\eta$  is finite if and only if $  \Sel\big(\Q(\mu_{p^\infty}), A_{f_k}[\varpi]\big)^\eta$ is finite; 
but the latter only depends on the residual representation and all forms in the Coleman family have the same residual representation.

Now we show that~\ref{pt:thm_main_mu:fg} implies~\ref{pt:thm_main_mu:zero_somes}. By Theorem \ref{thm:specialization} and \eqref{eq:imp}, there exists a point $k$ such that
\[
\frac{\DFS}{J_k\DFS} \overset{s_k}{\longrightarrow} \Y(\Q(\mu_{p^\infty}),T_{f_k}^\star(1))^\iota
\]
has finite kernel and is surjective. Taking $\eta$-isotypic components,~\ref{pt:thm_main_mu:fg} implies that the quotient $\frac{\DFS^\eta}{J_k\DFS^\eta}$ is finitely generated over ${\AoK}/P_k \cong \OK$, and hence so is
the component $\Y(\Q(\mu_{p^\infty}),T_{f_k}^\star(1))^\eta$. Next we show that~\ref{pt:thm_main_mu:zero_somes} implies~\ref{pt:thm_main_mu:fg} By Nakayama's lemma, it suffices to show that $\frac{\DFS^\eta}{J_k\DFS^\eta}$  is a finitely generated $\OK$-module. But $\Ker(s_k)$ is a quotient of $\widehat{\Ker(h_k)}$ (where $h_k$ is as in diagram \eqref{fundamental_diagram}) which is a finitely generated $\OK$-module by the proof of Theorem \ref{thm:specialization}. 
\end{proof}

\begin{remark}\label{rem:conja}
Recall that Conjecture A of Coates and Sujatha deals with the case of elliptic curves defined over a number field $F$. It says that the $\mu$-invariant of the dual fine Selmer group over the cyclotomic $\Z_p$-extension of $F$ is trivial  (see \cite[p.~822]{CoatesSujatha_fineSelmer}). This is equivalent to saying that this  dual fine Selmer group is a finitely generated $\Z_p$-module. Hence the fact that  $\Y(\Q(\mu_{p^\infty}),T_{f_k}^\star(1))^{\eta}$ is a finitely generated $\OK$-module (namely, part ~\ref{pt:thm_main_mu:zero_somes} of Theorem~\ref{thm:main_mu}) is essentially  Conjecture A  generalized for modular forms. Explicit examples of elliptic curves and modular forms with good \textit{ordinary} reduction at the prime $p$ satisfying Conjecture A are given in \cite{Chandrakant14}. For another approach to Coates and Sujatha's Conjecture A and variations of Iwasawa invariants beyond the ordinary case, see~\cite{NucSuj21}; see~\cite{KunNucSuj22} for its relations with Greenberg's Generalized Conjecture.
\end{remark}
\begin{theorem}\label{thm:lambda}
Assume any of the equivalent conditions in Theorem \ref{thm:main_mu}.  Then the $\lambda$-invariants of $\Y(\Q(\mu_{p^\infty}),T_{f_k}^\star(1))^{\eta}$ are equal for all but finitely many points $k\in {\Zf}$.
\end{theorem}
\begin{proof}
By Proposition~\ref{prop:torsion}, $\DFS$ is torsion as an $\tL$-module but it may not be torsion over ${\AoK}$. 
Write $H$ for the characteristic ideal of the ${\AoK}$-torsion submodule of $\DFS^{\eta}$, and factor it as $H=\prod_{1 \leq i \leq d}I_i^{m_i}$ where the $I^i$ are height one prime ideal of $\DFS^{\eta}$, and $m_i\geq 1$. Let
\[
\Sigma=\{I_1,\cdots , I_d\} \cup \{P_k \text{ such that } \Ker(s_k) \text{ is infinite}\}.
\]
Let $k$ be such that $P_k$ is not in $\Sigma$. The $\lambda$-invariant of $\Y(\Q(\mu_{p^\infty}),T_{f_k}^\star(1))^{\eta}$ is the $\Z_p$-rank of  $\Y(\Q(\mu_{p^\infty}),T_{f_k}^\star(1))^{\eta}$. By \eqref{eq:imp} and by the choice of $k$, we have a surjective map
\[
\frac{\DFS^\eta}{J_k\DFS^\eta} \overset{s_k}{\longrightarrow} \Y(\Q(\mu_{p^\infty}),T_{f_k}^\star(1))^{\iota,\eta}
\]
whose kernel is finite. Therefore, the  $\lambda$-invariant of $\Y(\Q(\mu_{p^\infty}),T_{f_k}^\star(1))^{\eta}$ is the $\Z_p$-rank of $\frac{\DFS^\eta}{J_k\DFS^\eta}$. We will show that the $\Z_p$-rank of $\frac{\DFS^\eta}{J_k\DFS^\eta}$ is equal to the ${\AoK}$-rank of the $(\tL)^\eta$-module $\DFS^{\eta}$ and that the latter is independent of $k$. The theorem will then follow.
	
Consider $\DFS^\eta$ as a module over ${\AoK}$.  By the structure theorem of modules over Iwasawa algebras, we can decompose 	$\DFS^\eta$ into an  ${\AoK}$-free part of a certain rank $r$ and an ${\AoK}$-torsion part of the form $\oplus_{i=1}^m {\AoK}/I_i^{m_i}$.  Recall that $P_k$ is a height one prime ideal of ${\AoK}$; it is the kernel of the evaluation map ${\AoK} \rightarrow \OO$ at $k$.  Also, note that we have chosen $k$ such that $P_k$ is not in $\Sigma$.  Therefore, $P_k$ is relatively prime to $I_i^{m_i}$ and hence ${\AoK}/(I_i^{m_i},P_k)$ is finite and hence has $\OK$-rank equal to $0$. Therefore, the $\OK$-rank of $\frac{\DFS^\eta}{J_k\DFS^\eta}$ is $r$, which is the ${\AoK}$-rank of $\DFS^\eta$.
\end{proof}

\section{Examples}
In this section, we sketch how to find examples satisfying the conditions of Theorems~\ref{thm:main_mu} and~\ref{thm:lambda}. We thank one of the referees for suggesting that we insert some numerical examples, and Nicolas Billerey for some discussions concerning them.

If there is any weight $k\in {\Zf}$ such that the residual representation of $\rho_{f_k}$ restricted to $G_{\Q_p}$ of the specialization $f_k$ of the Coleman family at the weight $k$ is irreducible, then the residual representation attached to the whole family $\boldsymbol{\rho}$ restricted to $G_{\Q_p}$ is irreducible and hence Hypothesis~\ref{hyp} is satisfied. Now, \cite[Theorem~2.6]{Edixhoven1992weight} says that the  residual representation of $\rho_{f_k}$ restricted to $G_{\Q_p}$  is irreducible if $f_k$ is a cuspidal eigenform whose $p$-th Fourier coefficient $a_p$ is $0$ in $\overline{\mathbb{F}}_p$ and $2 \leq k <p+1$.
Hence there are plenty of examples satisfying Hypothesis~\ref{hyp}. Consider the curve $y^2=x^3+32x+212$ (Cremona label or LMFDB label 140B, Antwerp label 140C) or the curve $y^2+xy=x^3-x^2-22x+884$ (Cremona label 182E, LMFDB label 182B, Antwerp label  182I). 
 Both of these curve  have supersingular reduction at the prime $p=3$. For both of these curves $a_3=3$. Their associated modular forms are weight $2$ cuspidal modular forms (cf. \cite{LMFDB}, and for the Antwerp labeling, see \cite[Table 1]{BirchKuyk}).

The next important hypothesis that we have made in Theorem \ref{thm:lambda} is a generalization of Conjecture A (see Remark \ref{rem:conja}).
In the following, we will explain how to find examples when Conjecture A is true for elliptic modular forms at supersingular primes.   
Suppose $p$ is a prime of good supersingular reduction for an elliptic curve $E$ over $\Q$. Note that the dual fine Selmer group of $E$ over the cyclotomic extension of $\Q$ is a quotient of the dual  signed Selmer groups which are torsion modules over the cyclotomic Iwasawa algebra. Recall also that the $\mu$-invariant is additive along a short exact sequence of finitely generated and torsion Iwasawa modules. Hence if one of the dual signed Selmer group over the cyclotomic extension of $\Q$ has $\mu$-invariant zero, then it is so for the dual fine Selmer group as well; hence Conjecture A is valid in those examples. 
Pollack in \cite{PollackTables} has already computed Iwasawa invariants of analytic  signed $p$-adic $L$-functions.  Recall that for CM elliptic curves over $\Q$ and for primes $p>3$ of supersingular reduction, the signed Iwasawa main conjectures have been proved by Pollack and Rubin in \cite{PollackRubin2004}. Now suppose that $E$ is an elliptic curve over $\Q$ without complex multiplication, with supersingular reduction at $p$, and such that the $p$-adic representation $\Gal(\overline{\Q}/\Q) \rightarrow \operatorname{GL}_{\Z_p}(T)$ on the automorphism group of the $p$-adic Tate module $T$ is surjective. Let $*\in\{+,-\}$ when $a_p=0$ and let $*\in\{\sharp,\flat\}$ if $a_p \neq 0$ (notations as in \cite{Kobayashi} and \cite{Sprung2012}). If $a_p\neq 0$ (by the Weil bound, this can occur only if $p=2,3$), choose  $*=\sharp$ or $\flat$ so that the $p$-adic $L$-function $L_p^*(E,X)$ is nonzero; both $L_p^\sharp(E,X)$ and $L_p^\flat(E/X)$ cannot be simultaneously zero (see~\cite[p. 1486]{Sprung2012}).
Under these assumptions, it is known that the characteristic ideal of $\mathcal{X}^*(E/\Q_\cyc)$ contains the ideal generated by $L_p^*(E,X)$; here $\mathcal{X}^*(E/\Q_\cyc)$ is 
the Pontryagin dual of the ($p$-primary) signed  Selmer group $\mathrm{Sel}^*(E/\Q_\cyc)$ over the cyclotomic $\Z_p$-extension of $\Q$ (see \cite[Theorem 4.1]{Kobayashi} for the case $a_p=0$ and see \cite[Theorem 1.4]{Sprung2012} for the case $a_p\neq 0$).
So if the $\mu$-invariant of $L_p^*(E,X)$ is trivial, then it is so for that of $\mathcal{X}^*(E/\Q_\cyc)$.
 For example, consider the examples of elliptic curves mentioned in the previous paragraph, i.e. with Cremona labels 140B and 182E. 
  For each of these curves, the  associated mod-$3$ Galois representation $G_{\Q} \rightarrow  \mathrm{Aut}(E[3]) \cong GL_2(\mathbb{F}_3)$ is surjective (see \cite[Theorem 3.2]{ReverterVila} which uses Antwerp labeling) and hence the $3$-adic representation $\Gal(\overline{\Q}/\Q) \rightarrow \operatorname{GL}_{\Z_3}(T)$ is also surjective (see \cite[p. 261]{Serre1972proprietes}). Also note that  both the $\mu$-invariants of $L_3^\sharp(E,X)$
and $L_3^\flat(E,X)$ are zero (see \cite{LMFDB} or \cite{PollackTables}).

\begin{remark}
	If we don't assume that the associated $p$-adic Galois representation is surjective then \cite[Theorem 4.1]{Kobayashi} or \cite[Theorem 1.4]{Sprung2012} will not imply the triviality of the $\mu$-invariant of $\mathcal{X}^*(E/\Q_\cyc)$ from that of the triviality of the $\mu$-invariant of $L_p^*(E,X)$. This is because their theorem will give that the characteristic ideal of $\mathcal{X}^*(E/\Q_\cyc)$ will contain the ideal generated by $p^nL_p^*(E,X)$ for some $n \geq 0$. 
\end{remark}

\bibliographystyle{amsalpha}
\bibliography{biblio}
\end{document}

%% file: preamble.tex
\usepackage[british]{babel}
\usepackage[all]{xy}
\usepackage{amscd}
\usepackage{amssymb}
\usepackage{amsthm}
\usepackage{enumitem}
\usepackage{mathrsfs,bbm}
\usepackage{xcolor,graphicx}
\usepackage{graphics}
\usepackage{soul}
\usepackage{comment}
\usepackage[all]{xy}
\usepackage{amscd}
\usepackage{amssymb,amsmath,latexsym}
\usepackage{amsthm}
\usepackage{enumitem}
\usepackage{mathrsfs,bbm}
\usepackage{dsfont}
\usepackage{tikz-cd}
\usepackage[T1]{fontenc}
\usepackage[utf8]{inputenc}  
\usepackage[colorlinks=true]{hyperref}

\newcommand{\R}{\mathscr{A}}
\newcommand{\X}{\mathscr{X}}
\newcommand{\Xf}{\mathscr{X}_{(k_0 ,i)}[r_0]}
\newcommand{\Xfr}{\mathscr{X}_{(k_0,i)}[r]}

\newcommand{\C}{\mathbb{C}_p}

\newcommand{\B}{\mathscr{B}}
\newcommand{\dualD}{\mathfrak{D}}

\newcommand{\W}{\mathscr{W}_N}
\newcommand{\Ao}{\mathbf{A}^\circ}
\newcommand{\AoK}{\mathbf{A}^\circ_{\K}}

\newcommand{\OOO}{\mathscr{A}^\circ}
\newcommand{\K}{\mathcal{K}} 
\newcommand{\OK}{\mathcal{O}_{\K}}

\newcommand{\T}{\mathscr{T}}
\newcommand{\TT}{\mathbf{T}}

\newcommand{\GS}{G_{\mathbb{Q},S}}
\newcommand{\Rf}{R_{(k_0 ,i)}[r_0]}

\newcommand{\Zf}{Z_{(k_0 ,i)}[r_0]}
\newcommand{\Zfr}{\mathscr{Z}_{(k_0 ,i)}[r]}
\newcommand{\ZFf}{\mathscr{Z}_{(k_0 ,i)}[r_0]}
\newcommand{\ZFfr}{\mathscr{Z}_{(k_0 ,i)}[r]}
\newcommand{\ZF}{\mathscr{Z}}